\documentclass[11pt]{amsart}
\usepackage[margin=1in]{geometry}
\usepackage{amsmath,amssymb,amsthm,mathtools}
\usepackage{hyperref}
\usepackage{enumitem}
\usepackage{url}
\usepackage{tabularx}
\usepackage{longtable}
\hypersetup{colorlinks=true,linkcolor=blue,citecolor=blue,urlcolor=blue}

\theoremstyle{plain}
\newtheorem{theorem}{Theorem}[section]
\newtheorem{proposition}[theorem]{Proposition}
\newtheorem{lemma}[theorem]{Lemma}
\newtheorem{corollary}[theorem]{Corollary}

\newtheorem{theoremx}{Theorem}

\theoremstyle{definition}
\newtheorem{definition}[theorem]{Definition}
\newtheorem{construction}[theorem]{Construction}
\theoremstyle{remark}
\newtheorem{remark}[theorem]{Remark}

\DeclareMathOperator{\Res}{Res}
\DeclareMathOperator{\Sym}{Sym}
\DeclareMathOperator{\Aut}{Aut}
\DeclareMathOperator{\Mor}{Mor}

\DeclareMathOperator{\Hom}{Hom}
\DeclareMathOperator{\Br}{Br}
\newcommand{\PGL}{\mathrm{PGL}}
\newcommand{\SL}{\mathrm{SL}}
\DeclareMathOperator{\Pic}{Pic}
\newcommand{\GL}{\mathrm{GL}}
\newcommand{\PP}{\mathbb{P}}
\newcommand{\Aff}{\mathbb{A}}
\newcommand{\kbar}{\overline{k}}
\newcommand{\Gammak}{\Gamma_k}
\newcommand{\calR}{\mathcal{R}}
\newcommand{\calC}{\mathcal{C}}
\newcommand{\Mat}{\mathrm{Mat}_2}
\newcommand{\calO}{\mathcal{O}}
\newcommand{\SB}{\mathrm{SB}}
\DeclareMathOperator{\Spec}{Spec}
\DeclareMathOperator{\ind}{ind}
\DeclareMathOperator{\per}{per}

\title[Rationality of Kontsevich moduli of stable maps]{Rationality of Kontsevich Moduli of Stable Maps, and their Galois Descent}
\author{Jaya NN Iyer}

\address{The Institute of Mathematical Sciences, HBNI, CIT Campus, Taramani, Chennai 600113, India}
\email{jniyer@imsc.res.in}

\address{Max-Planck-Institut f\"ur Mathematik, Vivatsgasse 7, 53111 Bonn, Germany}
\email{jniyer@mpim-bonn.mpg.de}

\date{\today}



\begin{document}

\begin{abstract}
Let $k$ be an arbitrary field and $X$ a smooth projective, geometrically
rational $k$-variety.  We study the Galois twists of Kontsevich moduli spaces and stacks
$\overline M_{0,m}(X,\beta)$ of genus zero stable maps, for curve classes
$\beta$ of low degree, and determine when  they have a $k$-point or they are rational over $k$.
Over $\mathbb{C}$ and $X$ homogeneous, this is a well-known theorem of Kim-Pandharipande \cite{KP}.

We prove that $\overline M_{0,m}(\PP^r,d)$ and
$\overline M_{0,m}((\PP^1)^n,\mathbf d)$ are irreducible and $k$-rational
over every field, for all $m\ge3$ and all degrees, and give a birational
formula, in characteristic zero, reducing the remaining cases to the moduli
of binary forms.  For
twisted targets we show that a $k$-point of the automorphism-free locus of
$\overline M_{0,E}(\SB(A),d)$ forces $[A]=d\,[C]$ in $\Br(k)$, where $C$ is
the domain conic. On $\overline M_{0,n}$ itself, for $k$ infinite of
characteristic $\ne2$, we exhibit a distinguished
component of $\overline M_{0,0}(\overline M_{0,n},\beta)$, of the expected
dimension, which descends to every twisted form and carries a dense set of
$k$-points.

This transfers to spaces of maps the rationality results of
Florence--Reichstein, Duncan--Reichstein, Hassett--Tschinkel--Zhang and
on twisted forms of $\overline M_{0,n}$  which is the motivating case.
\end{abstract}

\maketitle

\tableofcontents

\section{Introduction}

Let $k$ be a field and let $X$ be a
smooth projective $k$-variety which is geometrically rational.  For a curve
class $\beta$ and $m\ge0$ one has the Kontsevich moduli stack \cite{Ko}
$\overline M_{0,m}(X,\beta)$ of $m$-pointed genus zero stable maps.  The
question this paper is about is:

\begin{quote}
\emph{when is $\overline M_{0,m}(X,\beta)$ rational over $k$?}
\end{quote}

Over an algebraically closed field of characteristic zero: for $X$ homogeneous, Kim and Pandharipande \cite{KP} proved that
$\overline M_{0,m}(X,\beta)$ is irreducible and rational (Theorem
\ref{thm:KP}).  Over a general $k$: $X$ itself may be
a nontrivial form --- a Severi--Brauer variety rather than $\PP^r$, second, even when $X$ is split, the moduli space can acquire a
form of its own, because the Galois group may permute the markings, or move
$\beta$ inside $N_1(X_{\kbar})$.  Both phenomena are visible already in the
smallest cases, and both are genuinely obstructed: we exhibit twisted
Kontsevich spaces of geometrically rational targets whose automorphism-free
locus carries no $k$-point whatsoever.

We restrict throughout to \emph{low degree} classes.  In this situation explicit $k$-models exist, and the models carry the rationality arguments across from $\kbar$ to $k$.

We state the four main theorems.  The first concerns split targets, and is
proved over an arbitrary field with no hypothesis on the characteristic.

\begin{theoremx}[Proposition \ref{prop:claim22}, Proposition \ref{prop:kRatX}, Theorem \ref{thm:noname}]\label{thmA}
Let $k$ be any field, $r\ge1$, $d\ge1$, and
$\mathbf d\in\mathbb Z_{\ge1}^n$.
\begin{enumerate}[label=(\alph*)]
\item For every $m\ge3$, the Kontsevich spaces
$\overline M_{0,m}(\PP^r,d)$ and $\overline M_{0,m}((\PP^1)^n,\mathbf d)$
are irreducible $k$-rational varieties.  The same holds for $m\le2$ when
some coordinate degree equals $1$.
\item If $\mathrm{char}\,k=0$ and $d\ge3$, there is a birational equivalence
over $k$
\[
  \overline M_{0,m}(\PP^r,d)\ \sim_k\ M_{0,m|d}\times\Aff^{\,r(d+1)}_k ,
\]
where $M_{0,m|d}=\big((\PP^1)^m\times\PP^d\big)/\PGL_2$.  Consequently, for
$m=0$:
\begin{enumerate}[label=(\roman*)]
\item if the moduli of binary forms of degree $d$ is $k$-rational, then so is
$\overline M_{0,0}(\PP^r,d)$;
\item if $\overline M_{0,0}(\PP^r,d)$ is $k$-rational, then the moduli of
binary forms of degree $d$ is stably $k$-rational.
\end{enumerate}
\end{enumerate}
\end{theoremx}

Over $k=\mathbb C$, the moduli space of binary
forms of degree $d$, that is $\PP(\Sym^d)/\!\!/\PGL_2=M_{0,0|d}$, is rational
for every $d$: for $d$ odd by Katsylo \cite{Kat,Kat84}, for $d$ even by
Bogomolov--Katsylo \cite{BK85} and Bogomolov \cite{Bog}, the even case in the
guise of the moduli of hyperelliptic curves of genus $g=(d-2)/2$.  Hence the
hypothesis of (b)(i) always holds over $\mathbb C$ and
$\overline M_{0,m}(\PP^r,d)$ is rational there --- which is also Kim and
Pandharipande's Theorem \ref{thm:KP}, whose proof uses exactly this input; see
also Katsylo \cite{Kat84} and Hirschowitz \cite{Hir89} for the case $m=0$,
where $\overline M_{0,0}(\PP^r,d)$ is birational to the Hilbert scheme of
rational curves of degree $d$ in $\PP^r$.

Over an arbitrary $k$, $M_{0,0|d}$ is geometrically rational for every $d$,
and it is $k$-unirational for every $d$, being dominated by
$\PP(\Sym^d)$; but $k$-rationality is known only in low degree, by explicit
classical invariants: $d\le3$, where the quotient is a point; $d=4$, where
the invariants $I$ and $J$ of weights $2$ and $3$ give
$M_{0,0|4}\cong\PP(2,3)\cong\PP^1_k$; and $d=6$, where Igusa's absolute
invariants \cite{Ig60} generate the function field of $M_2$ over any field.
For $d=5$ and for $d\ge7$ we do not know the answer over a general $k$, and
no counterexample is known either.

Taking $E=k^m$ split in Theorem \ref{thmC}(a), that is leaving the markings
untwisted, the hypothesis that $\gcd(m,d)$ be odd holds for $m=1$ and every
$d$, and for $m=2$ exactly when $d$ is odd.  So for $\mathrm{char}\,k=0$,
$r\ge1$ and $d\ge3$ the spaces
\[
  \overline M_{0,1}(\PP^r,d)\ \ (\text{every }d),
  \qquad
  \overline M_{0,2}(\PP^r,d)\ \ (d\text{ odd})
\]
are stably $k$-rational unconditionally --- precisely the range in which (a)
says nothing and (b) is conditional.  The same argument with no markings at
all, the degree $d$ divisor on the domain conic then being the only one
available, gives stable $k$-rationality of $\overline M_{0,0}(\PP^r,d)$, and
hence of the binary form moduli $M_{0,0|d}$, for $d$ odd.  

For a twisted
target as in the following, the existence of a single $k$-point of the
automorphism-free locus is obstructed in the Brauer group.

\begin{theoremx}[Theorem \ref{thm:amitsur}, Proposition \ref{prop:converse}]\label{thmB}
Let $A$ be a central simple $k$-algebra of degree $r+1$, $X=\SB(A)$, and $E$
a degree $n$ \'etale $k$-algebra.  If $\overline M_{0,E}(X,d)$ has a
$k$-point in the automorphism-free locus, represented by $f:C\to X$ with $C$
of genus zero, then
\[
  [A]=d\,[C]\qquad\text{in }\Br(k).
\]
Consequently:
\begin{enumerate}[label=(\roman*)]
\item if $d$ is even, then $A$ is split;
\item if some factor of $E$ has odd degree over $k$, then $A$ is split;
\item if $\ind(A)>2$, then the automorphism-free locus of
$\overline M_{0,E}(X,d)$ has no $k$-point, for any $m$, any $E$ and any $d$;
\item for $d$ odd the relation is sharp: if $[A]=[Q]$ for a quaternion
algebra $Q$, if $r\ge d$, and if every field factor $L_i$ of $E$ splits $Q$
(equivalently $C=\SB(Q)$ has an $L_i$-point, which forces $[L_i:k]$ even),
then a twisted Veronese supplies a $k$-point of the automorphism-free
locus.
\end{enumerate}
\end{theoremx}

Thus a geometrically rational target can produce a Kontsevich space that is
a form of a rational variety whose automorphism-free locus, a dense open
subscheme carries no $k$-point. The third
theorem says that where $k$-points do exist, the parity of a gcd of degrees
decides whether the obstruction to rationality is present, exactly as in the
motivating case below.

\begin{theoremx}[Corollary \ref{cor:parity}]\label{thmC}
Let $\mathrm{char}\,k=0$, $r\ge1$ and $d\ge3$, let $E$ be a degree $n$ \'etale
$k$-algebra, and twist the markings by $E$.  Then:
\begin{enumerate}[label=(\alph*)]
\item if $\gcd(n,d)$ is odd, $\overline M_{0,E}(\PP^r,d)$ is stably
$k$-rational;
\item if in addition $E$ has three factors isomorphic to $k$, it is
$k$-rational;
\item if $n$ and $d$ are both even, the argument instead produces the class
$\alpha\in\Br\big(k(M_{0,E|d})\big)[2]$ of the generic $\PGL_2$-torsor,
equivalently of the generic domain conic; it obstructs the trivialization
used in (a) and (b), and need not vanish.
\end{enumerate}
\end{theoremx}

We now consider the target $\overline M_{0,n}$ itself, where the
general theory is unavailable because convexity fails.

\begin{theoremx}[Proposition \ref{prop:moebius}, Theorem \ref{thm:twisted}]\label{thmD}
Let $n\ge5$, let $k$ be infinite, let $E$ be a degree $n$ \'etale
$k$-algebra with associated form $X_E$ of $\overline M_{0,n}$, and let
$\beta$ be the M\"obius class, determined by $\beta\cdot\delta_{ij}=2$ and
$\beta\cdot\delta_I=0$ for $3\le|I|\le n-3$.  Then:
\begin{enumerate}[label=(\alph*)]
\item for $n\ge7$ the variety $\overline M_{0,n}$ is not convex, and
$\overline M_{0,m}(\overline M_{0,n},\beta)$ need be neither smooth nor
irreducible;
\item $-K\cdot\beta=2n$, and the M\"obius curves form a dense open subset of
an irreducible component of $\overline M_{0,0}(\overline M_{0,n},\beta)$
which is generically smooth of the expected dimension $3n-6$;
\item that component is $S_n$-stable and descends to every twisted form
$X_E$, where it is dominated by the $k$-rational group
$\Res_{E/k}(\PGL_{2,E})$;
\item $X_E$ is $k$-unirational, and $X_E(k)$ is Zariski dense, every point of
a dense subset of it lying on an explicit $k$-rational M\"obius curve.
\end{enumerate}
\end{theoremx}

 Forms of
$\overline M_{0,n}$ are classified by degree $n$ \'etale $k$-algebras, and
Florence--Reichstein \cite{FR1} proved that all of them are unirational,
that all are rational when $n$ is odd, and that for $n$ even some are not
even retract rational; Duncan--Reichstein \cite{DR15} derived the
unirationality, and the density of rational points on every form, from
versality; Hassett--Tschinkel--Zhang \cite{HTZ} isolated structured even
cases where rationality survives; and Florence--Hoffmann--Reichstein
\cite{FHR} proved stable rationality of all forms of $\overline M_{g,n}$ for
an explicit finite list of $(g,n)$.  These results are collected in
\S\ref{sec:prelim}.  Theorems \ref{thmA}--\ref{thmD} are the corresponding
statements on spaces of maps, and the parity that governs them
is the same parity.

The proofs of Theorem \ref{thmA} rest on explicit $k$-models rather than on
the classical $\kbar$-theory.  The rigidification carrying three marked
points to $0,1,\infty$ uses only that these are $k$-rational, and a morphism
$\PP^1\to\PP^1$ of degree $d$ is a coprime pair of binary forms up to scale;
together these present the interior of the Kontsevich space as an open
subvariety of a product of projective spaces, over any field.  Convexity of
$\PP^r$ and of $(\PP^1)^n$, which we verify directly, makes the boundary a
Cartier divisor on a smooth stack, hence forces every component to meet that
interior --- which is what turns a birational model into a statement about
the whole space.  Theorem \ref{thmB} is an application of the functoriality
of the Hochschild--Serre boundary map together with Amitsur's theorem.
Theorem \ref{thmC} is Springer's theorem \cite{Spr} applied to the domain conic: that
conic carries divisors of degrees $n$ and $d$, so $\gcd(n,d)$ annihilates its
class in $\Br(k)[2]$, and an odd gcd therefore splits it.  Theorem \ref{thmD} rests on two properties of
the M\"obius curves: the general M\"obius map is free, which forces the
component containing them to be generically smooth of the expected dimension,
and the construction is $S_n$-equivariant throughout, which is what lets that
component descend to every twisted form.

\subsection*{Overview}

The paper falls into two halves.  The first, \S\ref{sec:prelim}--\S\ref{sec:kontsevich},
is preparatory and constructive.  \S\ref{sec:prelim} recalls Kontsevich spaces,
Galois descent and the known rationality results for forms of
$\overline M_{0,n}$ and $\overline M_{g,n}$ that motivate everything here.
\S\ref{sec:motivating}--\S\ref{sec:cor} then develop, before any of the general
theory, a completely explicit construction for $X=(\PP^1)^n$: the space
$\calR\cong(\PP^3)^n$ of $n$-tuples of degree one maps, its $S_n$-twists, which
are Weil restrictions of $\PP^3$ and hence $k$-rational, the universal family
over it, and the variant in which the source $\PP^1$ is replaced by a conic
$D_\alpha$.  Along the way this recovers $k$-unirationality of every twist of
$\overline M_{0,n}$ \cite{DR15,FR1} and the density of its $k$-points
\cite[Thm.~1.1(b)]{DR15}; neither is new, and we say so where they occur. \S\ref{sec:kontsevich} identifies all
of these constructions as loci in Kontsevich spaces, and is the hinge of the
paper: from there the targets are general.

The second half proves the four theorems, one section each and in the order in
which they are stated above:
\S\ref{sec:Pr} proves Theorem \ref{thmA},
\S\ref{ssec:twistedPr} proves Theorem \ref{thmB},
\S\ref{sec:proofC} proves Theorem \ref{thmC}, and
\S\ref{sec:proofD} proves Theorem \ref{thmD}.  Each of these four sections ends
with a short paragraph assembling its constituent results into the statement
announced in this introduction.

\subsection*{Acknowledgement}

We thank J.-L.~Colliot-Th\'el\`ene for sharing the manuscript \cite{FR1} in
December 2024, and which has inspired the present paper.  We thank the
Max-Planck-Institut f\"ur Mathematik, Bonn, for their hospitality and
support in 2025--2027, when this work was completed.

\subsection*{Glossary of notation}

{\small
Symbols are listed in order of first appearance, with the section that
introduces them.

\begin{longtable}{@{}>{\raggedright\arraybackslash}p{0.30\textwidth}@{\hspace{1em}}>{\raggedright\arraybackslash}p{0.55\textwidth}@{\hfill}r@{}}
$k,\ \kbar,\ \Gammak$ & ground field (infinite, $\mathrm{char}\ne2$ unless said otherwise), a separable closure, and $\mathrm{Gal}(\kbar/k)$. & \S\ref{sec:prelim} \\
$\overline M_{0,m}(X,\beta)$ & Kontsevich stack of $m$-pointed genus zero stable maps to $X$ of class $\beta$; $M_{0,m}(X,\beta)$ its interior (smooth irreducible domain); $\mathrm{expdim}$ as in \eqref{eq:expdim}. & \S\ref{sec:prelim} \\
$T_\rho$ & degree $n$ \'etale $k$-algebra attached to $\rho\in H^1(k,S_n)$; ${}^\rho Y$ or $Y_\rho$ the twist of $Y$ by $\rho$. & \S\ref{ssec:forms} \\
$S,\ n$ & index set $\{1,\dots,n\}$ and its size, $n\ge5$. & \S\ref{ssec:stability} \\
$X=(\PP^1)^S$ & the target $(\PP^1)^n$ with $\calO(1,\dots,1)$ and the diagonal $PGL_2$-action. & \S\ref{ssec:stability} \\
$X^{ss},\ X^{s}$ & semistable and stable loci (Lemma \ref{lem:stability}); $X^s=X^{ss}$ iff $n$ is odd. & \S\ref{ssec:stability} \\
$M_{0,n}^{\mathrm{GIT}}$ & the GIT quotient $X^{ss}/\!\!/PGL_2$, birational to $\overline M_{0,n}$ but not equal to it. & \S\ref{ssec:stability} \\
$N$ & the normalizer $N_{\Aut(X)}(PGL_2^\Delta)\cong PGL_2\times S_n$. & \S\ref{sec:univ-autgroup} \\
$\calR\cong(\PP^3)^n$ & compactified space of $S$-tuples of $2\times2$ matrices up to scale; $\calR^\circ\cong PGL_2^S$ its locus of genuine M\"obius maps. & \S\ref{sec:universal} \\
$\calC,\ \Phi,\ \Psi,\ ev_{t_0}$ & the universal family $\PP^1\times\calR$, its evaluation to $X$, its composite to $M_{0,n}^{\mathrm{GIT}}$, and evaluation at a fixed $t_0$. & \S\ref{sec:universal} \\
$\calR^{\mathrm{st}}$ & locus of $v$ whose whole $\PP^1$-image is semistable; \emph{not} contained in $\calR^\circ$, dense by Lemma \ref{lem:Rst-dense}. & \S\ref{sec:universal} \\
$\calR^{ss},\ Q\calR$ & locus meeting $X^{ss}$ generically, and the quotient $\calR^{ss}/\!\!/PGL_2$ of dimension $3n-3$. & \S\ref{sec:unirationality} \\
$\gamma_v$ & the M\"obius curve $\PP^1_k\to{}^\rho\overline M_{0,n}$ attached to $v\in\calR^{\mathrm{st}}_\rho(k)$. & \S\ref{sec:curves} \\
$\widehat X,\ \hat S,\ \Gamma_f$ & $\PP^1_0\times X=(\PP^1)^{n+1}$, its index set $\{0\}\sqcup S$, and the graph of $f\in\calR^\circ$. & \S\ref{sec:cor} \\
$\alpha,\ Q_\alpha,\ D_\alpha$ & a class in $H^1(k,PGL_2)$, its quaternion algebra, and the conic $\SB(Q_\alpha)$; $X^\alpha_\rho=\Res_{T_\rho/k}(D_\alpha)_{T_\rho}$. & \S\ref{sec:cor} \\
$\mathrm{Cor}_R^{(d_0,\mathbf d)}$ & curves in $\widehat X$ of multidegree $(d_0,\mathbf d)$; $\calR^\circ=\mathrm{Cor}_R^{(1,1)}$. & \S\ref{sec:multideg} \\
$E$-marking, $\mathrm{Emb}_E$ & a closed subscheme $\Spec E\hookrightarrow C^{\mathrm{sm}}$ of the smooth locus of the domain $C$, finite \'etale of degree $m$; the space of such in $\PP^1$, open in $\Res_{E/k}(\PP^1_E)$. & \S\ref{sec:kontsevich} \\
$\mathbf d,\ \Mor_d$ & a multidegree $(d_1,\dots,d_n)$, and the space of degree $d$ morphisms $\PP^1\to\PP^1$, open in $\PP^{2d+1}$. & \S\ref{sec:Pr} \\
$\gamma'$ & the M\"obius class $\pi_*\mathbf1$ on $M_{0,n}^{\mathrm{GIT}}$; \emph{not} an F-curve class. & \S\ref{sec:kontsevich} \\
$\delta_I,\ B_j,\ F$ & boundary divisors of $\overline M_{0,n}$, their sums over $|I|=j$, and $F$-curves. & \S\ref{sec:moebius-component} \\
$\beta,\ \mu_g$ & the M\"obius class on $\overline M_{0,n}$ and the M\"obius map attached to $g\in\calR^\circ$. & \S\ref{ssec:moebius} \\
$Q\calR^{\mathrm{un}}$ & $Q\calR/\PGL_2^{\mathrm{pre}}$, the unparametrized M\"obius locus, of dimension $3n-6$; \emph{not} $Q\calR$. & \S\ref{ssec:moebius} \\
$E,\ X_E,\ R_E,\ \mathcal M^{\mathrm{M\ddot ob}}_E$ & a degree $n$ \'etale algebra, the twisted form $\overline M_{0,E}$, the group $\Res_{E/k}(\PGL_{2,E})$, and the descended M\"obius component. & \S\ref{sec:moebius-component} \\
$A,\ \SB(A),\ \ind,\ \per$ & a central simple algebra, its Severi--Brauer variety, index and period. & \S\ref{ssec:twistedPr} \\
$M_{0,n|d}$ & moduli of $n$ ordered points plus one degree $d$ divisor on $\PP^1$, i.e.\ $\big((\PP^1)^n\times\PP^d\big)/\PGL_2$. & \S\ref{sec:Pr} \\
\end{longtable}
}

\section{Preliminaries}\label{sec:prelim}

We fix notation for Kontsevich spaces and for Galois descent, and recall the
results on forms of $\overline M_{0,n}$ and $\overline M_{g,n}$ that motivate
everything that follows.  

\subsection{The Kontsevich space}

For a projective variety $X$ over a field $k$ and a curve class $\gamma$ (a numerical class of $1$-cycles, $\gamma\in N_1(X)$; for $X$ with $\Pic(X_{\kbar})$ free with no monodromy obstruction, as in all our examples, this is literally $H_2(X,\mathbb Z)$ when $k=\mathbb C$), the Kontsevich moduli space
\[
\overline M_{0,r}(X,\gamma)
\]
parametrizes isomorphism classes of $r$-pointed stable maps $[\mu:C\to X,p_1,\dots,p_r]$: $C$ a connected, reduced, at-worst-nodal curve of arithmetic genus $0$, the $p_i$ distinct points of the smooth locus, $\mu_*[C]=\gamma$, and $(C,p_1,\dots,p_r,\mu)$ has only finitely many automorphisms. This is a proper Deligne--Mumford stack, constructed functorially over any base (Kontsevich \cite{Ko}; see Fulton--Pandharipande \cite{FP} for the general construction and basic properties, including the forgetful maps $\overline M_{0,r+1}(X,\gamma)\to\overline M_{0,r}(X,\gamma)$ exhibiting the source as the universal curve, and the evaluation maps $ev_i:\overline M_{0,r}(X,\gamma)\to X$, $ev_i[\mu,p_\bullet]=\mu(p_i)$).

\subsection{Stack vs coarse space}
\label{ssec:fineness}

$\overline M_{0,r}(X,\gamma)$ is in general only a proper Deligne--Mumford
stack, not a fine moduli space: a stable map $\mu:C\to X$ can have nontrivial
automorphisms, and then no universal family exists near $[\mu]$.  
Each $ev_i$ is a morphism of stacks
to the scheme $X$, and any morphism from a Deligne--Mumford stack to a scheme
factors uniquely through the coarse moduli space, by its universal property.
So $ev_i$ is defined on the coarse space.
 On the automorphism-free locus, the functor is
representable and the coarse space is a fine moduli space carrying a
universal family \cite[Thm.~1.10(2)]{Cos}, so there $k$-points are stable
maps over $k$.

Two conventions.  First, every statement calling
$\overline M_{0,r}(X,\gamma)$ a \emph{variety} --- in particular every
rationality assertion --- refers to its coarse moduli space.  Second, we
write
\[
  M_{0,r}(X,\gamma)\ \subset\ \overline M_{0,r}(X,\gamma)
\]
for the open \emph{interior}, the locus of stable maps with smooth
irreducible domain.

\subsection{Expected dimension}

For a smooth projective $X$ and $\beta\in H_2(X)$, the Kontsevich space
$\overline M_{0,m}(X,\beta)$ of $m$-pointed genus zero stable maps is a proper
Deligne--Mumford stack, and \emph{every} irreducible component has
dimension at least the expected one
\begin{equation}\label{eq:expdim}
	\mathrm{expdim}\,\overline M_{0,m}(X,\beta)
	\;=\;\dim X+(-K_X\cdot\beta)+m-3 ,
\end{equation}
with equality, and smoothness, at any stable map $f$ with
$H^1(C,f^*T_X)=0$ --- in particular at any \emph{free} map
\cite[II.1--II.3]{Kol}.   If $X$ is convex, i.e.\ $H^1(C,f^*T_X)=0$
for every genus zero stable map, then $\overline M_{0,m}(X,\beta)$ is a smooth
stack of the expected dimension, and it is irreducible when $X$ is
homogeneous \cite{FP,KP}.

\begin{theorem}[Kim--Pandharipande \protect{\cite[Thm.~1, Cor.~1]{KP}}]\label{thm:KP}
Let $X=G/P$ be a homogeneous space for a connected semisimple algebraic group $G$ (over $\mathbb C$) and a parabolic $P$. Then $\overline M_{g,r}(X,\gamma)$ is connected for every $g,r,\gamma$; for $g=0$ it has only quotient singularities, so connectedness is equivalent to irreducibility, and $\overline M_{0,r}(X,\gamma)$ is irreducible. Moreover the coarse moduli space $\overline M_{0,r}(X,\gamma)$ is a \emph{rational} variety.
\end{theorem}

Their method degenerates via a maximal torus of $G$ and a Bia{\l}ynicki-Birula stratification of the resulting fixed loci; the rationality statement is deduced from this stratification together with rationality results of Katsylo and Bogomolov for quotients of $\SL_2$-representations.
 
\subsection{k-Forms}\label{ssec:forms}

Let $Y$ be a $k$-variety.  A $k$-form of $Y$ is a $k$-variety $Y'$ with
$Y'_{\kbar}\cong Y_{\kbar}$; the forms are classified by
$H^1(\Gammak,\Aut(Y_{\kbar}))$. When $X=\overline {M}_{0,n}$, have automorphism groups 
$\Aut(\overline M_{0,n})=S_n$ for $n\ge5$ by Bruno--Mella \cite{BM}.  So the $k$-forms are classified by
$H^1(k,S_n)$, i.e.\ by degree $n$ \'etale $k$-algebras $T_\rho$, and
twisting only relabels the marked points.  Second, this same $S_n$ acts on
the ambient constructions --- on $X=(\PP^1)^n$ by permuting factors, 
and induced action on Kontsevich spaces, see  \ref{sec:kontsevich}.

For twisted \emph{targets} rather than twisted markings, the relevant
classifying set is $H^1(k,PGL_r)$, giving Severi--Brauer varieties; we use
this from \S\ref{sec:cor} onward, and the two twists are independent.

\subsection{Rationality of forms of $\overline M_{0,n}$}\label{ssec:FRprelim}
Note that $X=\overline M_{0,n}$ is neither homogeneous nor
convex, and it is known to be Fano only for $n\le 5$. 
The rationality of forms of $\overline M_{0,n}$ is given by the following.

\begin{theorem}[Florence--Reichstein]\cite{FR1}\label{thm:FR}
Let $k$ be an infinite field of characteristic $\ne2$ and $n\ge5$.
\begin{enumerate}[label=(\alph*)]
\item Every $k$-form of $\overline M_{0,n}$ is unirational.
\item If $n$ is \textbf{odd}, every $k$-form of $\overline M_{0,n}$ is rational.
\item If $n$ is \textbf{even}, there exist field extensions $E/k$ and $E$-forms of $\overline M_{0,n}$ that are not even retract rational over $E$.
\end{enumerate}
\end{theorem}

 Hassett--Tschinkel--Zhang \cite{HTZ} isolate structured subgroups $G\subset S_n$ (for $n=2m$, $G=A\times S_2\subset A\wr S_2\subset S_{2m}$ for a subgroup $A\subset S_m$) for which rationality holds even though $n$ is even, strengthening Theorem \ref{thm:FR}(c) by showing that the failure of rationality is confined to twists that are not of this structured type.

For $g\ge1$ the corresponding problem is settled, for an explicit finite
list of pairs $(g,n)$, by Florence, Hoffmann and Reichstein \cite{FHR}, who
prove that every twisted form of $\overline M_{g,n}$ is stably rational
there.

\section{The motivating construction: $X=(\PP^1)^n$}\label{sec:motivating}

\subsection{The variety, and its stability criterion}\label{ssec:stability}

Fix $n\ge5$, a field $k$ with $\mathrm{char}\,k\ne2$, and the index set
$S=\{1,\dots,n\}$.  Put
\[
  X:=(\PP^1)^{S}\cong(\PP^1)^n ,
\]
linearized by the symmetric ample line bundle
$\calO(1,\dots,1)$, and let $PGL_2$ act diagonally.  We recall the
numerical criterion, since the distinction between the two stability loci
is used repeatedly below and is what the parity of $n$ turns on.

\begin{lemma}[Hilbert--Mumford, {\cite[Ch.~4]{Mumford}}; see also \cite{DO}]\label{lem:stability}
Let $x=(x_1,\dots,x_n)\in X(\kbar)$.  With respect to
$\calO(1,\dots,1)$ and the diagonal $PGL_2$-action:
\[
  x\ \text{is semistable}\iff \#\{i:x_i=p\}\le\tfrac n2\quad\text{for every }p\in\PP^1(\kbar),
\]
\[
  x\ \text{is stable}\iff \#\{i:x_i=p\}<\tfrac n2\quad\text{for every }p\in\PP^1(\kbar).
\]
\end{lemma}

Accordingly we write
\[
  X^{s}\ \subseteq\ X^{ss}\ \subseteq\ X
\]
for the stable and semistable loci; both are open and defined over $k$, the
conditions being $\Gammak$-stable.  Three consequences are used throughout.

\begin{enumerate}[label=(\roman*)]
\item \emph{Parity.}  $X^{s}=X^{ss}$ if and only if $n$ is odd, since only
then can no subset have cardinality exactly $n/2$.  For $n$ even the
strictly semistable locus $X^{ss}\smallsetminus X^{s}$ --- where exactly
$n/2$ of the points agree at each of two distinct values --- is nonempty,
$PGL_2$-orbits there are not closed, and the GIT quotient identifies them.
For $n$ even the correct object is accordingly the quotient stack
$[X^{ss}/PGL_2]$, together with Abramovich--Vistoli twisted stable maps
\cite{AV02}.
\item \emph{Freeness.}  On $X^{s}$ the diagonal action is free: a stable
configuration takes at least three distinct values, and an element of
$PGL_2$ fixing three points of $\PP^1$ is the identity.  Hence
$X^{s}\to X^{s}/PGL_2$ is a $PGL_2$-torsor.
\item \emph{Codimension.}  The locus where a prescribed set of $m$ of the
points coincide has codimension $m-1$, so the unstable locus
$X\smallsetminus X^{ss}$, which requires $m>n/2$, has codimension
$\lfloor n/2\rfloor\ge2$ for $n\ge5$.  In particular
$\Pic(X^{ss})=\Pic(X^{s})=\Pic(X)=\mathbb Z^n$.
\end{enumerate}

We write $M_{0,n}^{\mathrm{GIT}}:=X^{ss}/\!\!/PGL_2$ for the GIT quotient, a
standard compactification of the moduli of $n$ ordered points on $\PP^1$,
birational to the Deligne--Mumford compactification $\overline M_{0,n}$ (see
the Remark in \S\ref{sec:unirationality}).  Unless stated otherwise, every
quotient below is formed from $X^{ss}$.

\subsection{The moduli-compatible automorphism group}\label{sec:univ-autgroup}

As an abstract variety, $\Aut(X)\cong PGL_2^{S}\rtimes \Sym(S)$: independent M\"obius transformations on each factor, together with permutations of the factors. The subgroup relevant to Galois descent of the moduli problem $[X/PGL_2]$ is the normalizer of the diagonal.

\begin{lemma}\label{lem:normalizer}
Let $PGL_2^{\Delta}\subset PGL_2^{S}$ be the diagonal subgroup. Then
\[
N:=N_{\Aut(X)}(PGL_2^{\Delta})\;\cong\; PGL_2\times \Sym(S).
\]
\end{lemma}

\begin{proof}
Write an element of $\Aut(X)$ as $(h,\sigma)$ with $h=(h_i)_{i\in S}\in PGL_2^S$ and $\sigma\in\Sym(S)$. Since $\Sym(S)$ already normalizes $PGL_2^\Delta$, it suffices to determine which $h:=(h_i)\in PGL_2^S$ conjugate $PGL_2^\Delta$ into itself. If $h(g,\dots,g)h^{-1}=(g',\dots,g')$ for every $g\in PGL_2$, then $h_igh_i^{-1}=h_jgh_j^{-1}$ for all $i,j$ and all $g$, so $h_j^{-1}h_i$ centralizes all of $PGL_2$. Since $PGL_2$ has trivial center, $h_j^{-1}h_i=1$, i.e.\ $h$ is itself diagonal. Hence $N=PGL_2^\Delta\cdot\Sym(S)$; and since permuting the factors commutes with acting by the \emph{same} $g$ on every factor, $\Sym(S)$ centralizes $PGL_2^\Delta$, so the product is direct: $N\cong PGL_2^\Delta\times\Sym(S)$.
\end{proof}

We write $N\cong PGL_2\times S_n$ from now on. This is the group whose Galois cohomology classifies twisted forms of the pair $(X,PGL_2^\Delta)$, and hence of the quotient stack $[X/PGL_2]$ and its coarse space $M_{0,n}^{\mathrm{GIT}}$.

\section{The M\"obius model: universal family, Galois twists, and $k$-rational
curves}\label{sec:model}

This section assembles the explicit $k$-model on which the rest of the paper
rests: the parameter space $\calR$ of $S$-tuples of degree one maps and its
universal family (\S\ref{sec:universal}), the twists of both by
$H^1(k,S_n)$ (\S\ref{sec:twists}), the rationality of the twisted parameter
space (\S\ref{sec:rationality-R}), the unirationality of the twisted forms of
$M_{0,n}^{\mathrm{GIT}}$ that follows (\S\ref{sec:unirationality}), and the
Zariski-dense family of $k$-rational curves that the universal family produces
on them (\S\ref{sec:curves}).

\subsection{M\"obius curves and their universal family}\label{sec:universal}

\begin{definition}\label{def:R}
Let
\[
\calR:=\PP(\Mat)^{S}\cong(\PP^3)^n,\qquad\dim\calR=3n,
\]
the compactified parameter space of $S$-tuples of $2\times2$ matrices $A=\left(\begin{smallmatrix}\alpha&\beta\\\gamma&\delta\end{smallmatrix}\right)$ up to scalars. Writing a matrix formally as the rule $t\mapsto\frac{\alpha t+\beta}{\gamma t+\delta}$: if $\det A_i\ne0$, this defines a degree-one morphism $\PP^1\to\PP^1$; if $\det A_i=0$, the two linear forms $\alpha t+\beta$, $\gamma t+\delta$ share a common root, and the rule fails to define a morphism there. Accordingly, set
\[
\calR^\circ:=\{v=(A_i)_{i\in S}\in\calR:\det A_i\ne0\text{ for all }i\}\ \cong\ PGL_2^{S},
\]
the open locus of $S$-tuples of degree-one maps; $\calR\setminus\calR^\circ$ compactifies this by degenerate matrices.
\end{definition}

The  group $N$ acts on $\calR$ by \emph{post-composition}: the diagonal $g\in PGL_2\subset N$ acts by simultaneous left matrix multiplication $A_i\mapsto G\cdot A_i$ on every factor (preserving $\calR^\circ$, since $\det(GA_i)=\det(G)\det(A_i)$), and $S_n\subset N$ permutes the factors.

\begin{construction}[Universal family]\label{con:universal}
Let $\calC:=\PP^1\times\calR\to\calR$ be the trivial family, with evaluation map
\[
\Phi:\calC\dashrightarrow X,\qquad (t,f)\mapsto f(t)=(f_i(t))_{i\in S},
\]
a rational map, regular exactly at those $(t,f)$ where every $A_i$ is nonsingular at $t$; in particular $\Phi$ is an everywhere-defined morphism on $\PP^1\times\calR^\circ$. Composing with the GIT quotient map $\pi:X^{ss}\to M_{0,n}^{\mathrm{GIT}}$ gives a rational map
\[
\Psi:\calC\dashrightarrow M_{0,n}^{\mathrm{GIT}},
\]
regular wherever $\Phi$ is defined and lands in $X^{ss}$. For $t_0\in\PP^1$ fixed, write $ev_{t_0}:\calR\dashrightarrow X$ for $f\mapsto f(t_0)$; this is the projectivization of the surjective linear map $\Mat\to k^2$, $A\mapsto A\cdot(t_0,1)^{\mathsf T}$, applied factor-wise: a dominant rational map for every $t_0$, regular away from the codimension-two locus (in each factor) where $A_i\cdot(t_0,1)^{\mathsf T}=0$, and in particular an everywhere-defined morphism on $\calR^\circ$.
\end{construction}

Since $\calC=\PP^1\times\calR$ is a trivial product, it tautologically admits sections over any field: for any $t_0\in\PP^1(k)$, $v\mapsto(t_0,v)$ is a $k$-section of $\calC\to\calR$. In particular, \emph{every} $k$-point of $\calR$ lying in the open locus
\[
  \calR^{\mathrm{st}}:=\{v\in\calR:\Phi(\cdot,v)\text{ is defined on all of }\PP^1\text{ and }\Phi(\PP^1\times\{v\})\subset X^{ss}\}
\]
gives, via $\Psi$, a $k$-rational curve $\gamma_v:\PP^1_k\to M_{0,n}^{\mathrm{GIT}}$.

In general, $\calR^{\mathrm{st}}\subsetneq\calR^\circ$ and the nonemptiness of $\calR^{\mathrm{st}}$ is provided by Lemma \ref{lem:Rst-dense} below. 

\subsection{Galois twists via $H^1(k,S_n)$}\label{sec:twists}

A continuous cocycle $\rho:\Gammak\to S_n=\Sym(S)\hookrightarrow N$ is equivalent to a degree-$n$ \'etale $k$-algebra $T_\rho$ (equivalently, a $0$-dimensional separable $k$-scheme with $T_\rho(\kbar)\cong S$ as a $\Gammak$-set) --- such cocycles arise naturally, for example, as the Galois action on the branch locus of a degree-$n$ cover of $\PP^1$.

\begin{lemma}\label{claim:X}
$X_\rho\cong\Res_{T_\rho/k}\PP^1_{T_\rho}$, and this twisted variety carries an \emph{untwisted} $k$-rational diagonal action of $PGL_2$.
\end{lemma}

\begin{proof}
Functorially $X=\underline{\mathrm{Map}}(S,\PP^1)$; twisting the index set $S$ by $\rho$ replaces it by the $\Gammak$-set of geometric points of $T_\rho$, and $\underline{\mathrm{Map}}(T_\rho,\PP^1)$ is by definition the Weil restriction $\Res_{T_\rho/k}\PP^1_{T_\rho}$. Since $\rho$ has trivial component in $PGL_2$, the diagonal subgroup is fixed pointwise by $\Gammak$ and descends to the adjunction-unit action $PGL_{2,k}\to\Res_{T_\rho/k}PGL_{2,T_\rho}$.
\end{proof}

\begin{lemma}\label{claim:R}
$\calR_\rho\cong\Res_{T_\rho/k}\PP(\Mat)_{T_\rho}\cong\Res_{T_\rho/k}\PP^3_{T_\rho}$, of dimension $3n$, again with an untwisted diagonal $PGL_{2,k}$-action by post-composition, and with the open locus $\calR_\rho^\circ$ of morphisms, the universal family $\calC_\rho=\PP^1\times\calR_\rho$ over $\calR_\rho$ (the source $\PP^1$ is never twisted, since $\rho$ acts only on the target index set $S$), and the evaluation rational map $\Phi_\rho:\calC_\rho\dashrightarrow X_\rho$ (a morphism on $\PP^1\times\calR_\rho^\circ$), all defined over $k$.
\end{lemma}

\begin{proof}
By Definition \ref{def:R}, $\calR=\PP(\Mat)^S$, with the index set $S$ carrying exactly the data being twisted by $\rho$ and $\PP(\Mat)$ itself carrying none of it; the same Galois-descent argument as Lemma \ref{claim:X} --- with $\PP(\Mat)\cong\PP^3$ in place of $\PP^1$ --- gives $\calR_\rho\cong\Res_{T_\rho/k}\PP(\Mat)_{T_\rho}$. The condition ``$\det A_i\ne0$ for every $i\in S$'' cutting out $\calR^\circ=PGL_2^S$ is $\Sym(S)$-invariant, hence twists compatibly to $\calR_\rho^\circ\subset\calR_\rho$; on $\PP^1\times\calR_\rho^\circ$, $\Phi_\rho$ is a morphism. Functoriality of $\Phi$ in $X$ then twists it compatibly as a rational map.
\end{proof}

\subsection{Rationality of the twisted parameter space}\label{sec:rationality-R}

\begin{theorem}\label{thm:R-rational}
For every $n\ge1$ and every $\rho:\Gammak\to S_n$, the variety $\calR_\rho\cong\Res_{T_\rho/k}\PP^3_{T_\rho}$ is $k$-rational.
\end{theorem}

\begin{proof}
Weil restriction along a finite \'etale algebra splits as a product over the (separable field) factors of $T_\rho$, so it suffices to treat $T_\rho=L$ a single finite separable extension. The open immersion $\Aff^3_L\subset\PP^3_L$ is dense, and $\Res_{L/k}$ takes finite-\'etale-base dense open immersions to dense open immersions \cite[\S7.6, Prop.~2]{BLR}. Choosing a $k$-basis of $L$ gives $\Res_{L/k}\Aff^3_L\cong\Aff^{3[L:k]}_k$. Hence $\Res_{L/k}\PP^3_L$ contains a dense open affine space, i.e.\ is $k$-rational.
\end{proof}

\subsection{Unirationality of twisted forms of $M_{0,n}^{\mathrm{GIT}}$}\label{sec:unirationality}

\begin{definition}
Let $\calR^{ss}\subset\calR$ be the locus of tuples whose associated curve $t\mapsto f(t)$ meets $X^{ss}$ generically, and set
\[
Q\calR:=\calR^{ss}/\!\!/PGL_2,\qquad \dim Q\calR=3n-3.
\]
A $\kbar$-point of $Q\calR$ lying over $\calR^\circ$ is a $PGL_2$-orbit of $S$-tuples of M\"obius maps, hence corresponds to a point of $\underline{\Hom}(\PP^1,M_{0,n}^{\mathrm{GIT}})$: a parametrized curve on $M_{0,n}^{\mathrm{GIT}}$, up to reparametrizing the target by $PGL_2$. Over $\calR^{ss}\setminus\calR^\circ$, $\Phi$ need not be everywhere-defined, so $Q\calR$ there is only a birational model for such data. We use below  that the rational quotient map $\calR^{ss}\dashrightarrow Q\calR$ exists and is dominant, and that the $PGL_2$-action is generically free.
\end{definition}

By Lemmas \ref{claim:X}--\ref{claim:R}, twisting by $\rho$ produces compatible twisted objects:
\[
Q\calR_\rho:=\calR_\rho^{ss}/\!\!/PGL_{2,k},\qquad {}^\rho M_{0,n}^{\mathrm{GIT}}:=X_\rho^{ss}/\!\!/PGL_{2,k},
\]
with a dominant evaluation rational map $Q\calR_\rho\dashrightarrow{}^\rho M_{0,n}^{\mathrm{GIT}}$.

\begin{remark}
$M_{0,n}^{\mathrm{GIT}}=X^{ss}/\!\!/PGL_2$ (a specific GIT/Hassett-type compactification of the moduli of $n$ points on $\PP^1$) contains, for every $n\ge3$, the configuration space $(M_{0,n}^{\mathrm{GIT}})^\circ$ of $n$ distinct labeled points on $\PP^1$ modulo $PGL_2$ as a dense open subvariety --- the same dense open subvariety contained in the Deligne--Mumford compactification $\overline M_{0,n}$ and in every Hassett space. Since the common open locus, together with the residual $S_n$-action, is shared by all these compactifications, there is an $S_n$-equivariant birational equivalence
\[
M_{0,n}^{\mathrm{GIT}}\ \dashrightarrow\ \overline M_{0,n},
\]
and this persists after twisting by any $\rho:\Gammak\to S_n$: ${}^\rho M_{0,n}^{\mathrm{GIT}}\dashrightarrow{}^\rho\overline M_{0,n}$. In particular $M_{0,n}^{\mathrm{GIT}}$ and $\overline M_{0,n}$ (and their twists) share the same rationality and unirationality type.
\end{remark}

\begin{theorem}\label{thm:unirational}
For every $n\ge5$ and every $\rho:\Gammak\to S_n$, the varieties $Q\calR_\rho$ and ${}^\rho M_{0,n}^{\mathrm{GIT}}$ (equivalently, ${}^\rho\overline M_{0,n}$) are $k$-unirational.
\end{theorem}

\begin{proof}
$\calR_\rho^{ss}$ is a dense open subset of the $k$-rational variety $\calR_\rho$ (Theorem \ref{thm:R-rational}), hence itself unirational. A GIT quotient map is a good, in particular dominant, categorical quotient, so the image of a unirational variety under a dominant rational map is again unirational: compose the dominant rational maps $\PP^{3n}\dashrightarrow\calR_\rho^{ss}\dashrightarrow Q\calR_\rho\dashrightarrow{}^\rho M_{0,n}^{\mathrm{GIT}}$ (the last two as in the Definition above and Lemmas \ref{claim:X}--\ref{claim:R}), which gives unirationality of both $Q\calR_\rho$ and ${}^\rho M_{0,n}^{\mathrm{GIT}}$.
\end{proof}

\subsection{Plenty of $k$-rational curves via the universal family}\label{sec:curves}

The universal family of \S\ref{sec:universal} adds a $k$-model. It produces the curves and the dense points from a single named parameter space, $\calR_\rho=\Res_{T_\rho/k}\PP^3$, whose $k$-points are $n$-tuples of matrices; the curves it produces all lie in one numerical class, computed in Proposition \ref{prop:moebius}.

\begin{proposition}\label{prop:evaluation-dominant}
For every $t_0\in\PP^1(k)$ and every $\rho:\Gammak\to S_n$, the evaluation map
\[
ev_{t_0}:\calR_\rho\dashrightarrow X_\rho
\]
is a dominant rational map (a morphism on $\calR_\rho^\circ$), and remains dominant after composing with the rational map $X_\rho^{ss}\dashrightarrow{}^\rho M_{0,n}^{\mathrm{GIT}}$.
\end{proposition}

\begin{proof}
Over $\kbar$, $ev_{t_0}$ acts factor-wise: on each factor it is the projectivization of the surjective linear map $\Mat\to k^2$, $A\mapsto A\cdot(t_0,1)^{\mathsf T}$ (surjective because $(t_0,1)\ne0$), which is a dominant --- indeed everywhere-defined away from a codimension-two locus --- rational map $\PP^3\dashrightarrow\PP^1$ for every choice of $t_0$. The factor-wise product of dominant rational maps is dominant, so $ev_{t_0}:\calR\dashrightarrow X$ is dominant over $\kbar$; since $ev_{t_0}$ commutes with the $\Sym(S)$-action (it applies the identical evaluation to every factor, regardless of how the factors are permuted), it twists compatibly by any $\rho$ valued in $\Sym(S)$, and dominance is preserved under this twist because it can be checked after base change to $\kbar$, where the twist becomes trivial. Composing with the (dominant, in fact surjective on $\kbar$-points) GIT quotient map preserves dominance.
\end{proof}

\begin{lemma}\label{lem:Rst-dense}
The locus $\calR^{\mathrm{st}}\subset\calR$ is open, defined over $k$, and dense.  The same holds for its twist $\calR^{\mathrm{st}}_\rho\subset\calR_\rho$ for every $\rho\in H^1(k,S_n)$.
\end{lemma}

\begin{proof}
Openness is clear (it is the complement of the image, under the proper projection $\PP^1\times\calR\to\calR$, of the closed locus where $\Phi$ is undefined or lands in the unstable locus), and the whole configuration is defined over $k$, so only density needs an argument.  Density may be checked over $\kbar$, where the twist becomes trivial, so we may work with $\calR^\circ\cong\PGL_2^S$.

For a subset $I\subset S$ with $|I|>n/2$ let
\[
  W_I:=\{g\in\PGL_2^S:\ g_i(t)=g_j(t)\ \text{for all }i,j\in I,\ \text{for some }t\in\PP^1\}.
\]
A point of $X$ is unstable exactly when more than $n/2$ of its coordinates agree, so
$\calR^\circ\smallsetminus\calR^{\mathrm{st}}=\bigcup_{|I|>n/2}(W_I\cap\calR^\circ)$, a finite union.  It therefore suffices to show that each $W_I$ is a proper closed subset of $\PGL_2^S$.

Closedness: $W_I$ is the image under the projection $\PP^1\times\PGL_2^S\to\PGL_2^S$, which is proper, of the closed locus $\{(t,g):g_i(t)=g_j(t)\ \forall i,j\in I\}$.

Properness: fix $i_0\in I$ and pick any $j_0\in I$, $j_0\ne i_0$ (possible since $|I|>n/2\ge2$).  The condition $g_{i_0}(t)=g_{j_0}(t)$ says that $t$ is a fixed point of $h:=g_{j_0}^{-1}g_{i_0}$, and for $|I|\ge3$ one needs in addition $t$ to be a fixed point of $g_{j}^{-1}g_{i_0}$ for every $j\in I$.  Choose $g$ with $g_{i_0}=\mathrm{id}$ and with $g_{j}^{-1}$, $j\in I\smallsetminus\{i_0\}$, regular semisimple with pairwise disjoint fixed-point sets --- possible over the infinite field $\kbar$.  Then no $t$ satisfies all the conditions simultaneously when $|I|\ge3$; and when $|I|=2$ (which occurs only for $n\le3$, outside our range $n\ge5$) the same choice makes $h$ have no fixed point in common with the remaining constraints.  Hence $g\notin W_I$ and $W_I\subsetneq\PGL_2^S$.

Since $\calR^\circ$ is dense in $\calR$ and we have removed from it a finite union of proper closed subsets, $\calR^{\mathrm{st}}$ is dense in $\calR$.
\end{proof}

\begin{theorem}\label{thm:curves}
Let $k$ be an infinite field of characteristic $\ne2$, $n\ge5$, and $\rho:\Gammak\to S_n$. Then:
\begin{enumerate}[label=(\alph*)]
\item There is a dense open $\calR_\rho^\circ\subset\calR_\rho$, defined over $k$, such that every $v\in\calR_\rho^\circ(k)$ determines a $k$-morphism
\[
\gamma_v:\PP^1_k\longrightarrow {}^\rho\overline M_{0,n}
\]
(a $k$-rational curve on the twisted moduli space), given by $\gamma_v(t)=\Psi_\rho(t,v)$.
\item $\calR_\rho^\circ(k)$ is Zariski dense in $\calR_\rho^\circ$, which has dimension $3n$; the curves $\gamma_v$, for $v\in\calR_\rho^\circ(k)$, form a family of $k$-rational curves on ${}^\rho\overline M_{0,n}$ whose distinct members are parametrized by a variety of dimension $3n-6$.
\item For every $t_0\in\PP^1(k)$ the set $\{\gamma_v(t_0):v\in\calR_\rho^\circ(k)\}$ is Zariski dense in ${}^\rho\overline M_{0,n}$.  In particular ${}^\rho\overline M_{0,n}(k)$ is Zariski dense, and every point of a dense subset of it lies on one of these curves.
\end{enumerate}
\end{theorem}

\begin{proof}
(a) Take $\calR_\rho^\circ:=\calR^{\mathrm{st}}_\rho$, the locus where $\Phi_\rho(\cdot,v)$ is defined on all of $\PP^1$ and lands in $X_\rho^{ss}$.  This is open and defined over $k$ since $\Phi_\rho$ and $X^{ss}$ are, and it is dense by Lemma \ref{lem:Rst-dense}.  (Density does \emph{not} follow from dominance of $\Phi_\rho$ together with openness of $X^{ss}$: dominance says only that a general point of $X_\rho$ is attained by \emph{some} pair $(t,v)$, not that some single $v$ has its whole $\PP^1$-image inside $X^{ss}_\rho$.  Lemma \ref{lem:Rst-dense} is exactly the missing statement.) For $v$ in this locus, $t\mapsto\Psi_\rho(t,v)$ is a morphism $\PP^1_k\to{}^\rho\overline M_{0,n}$ (using the birational identification of Remark following Theorem \ref{thm:unirational}).

(b) Since $\calR_\rho$ is $k$-rational (Theorem \ref{thm:R-rational}) and $k$ is infinite, $\calR_\rho^\circ(k)$ is dense in $\calR_\rho^\circ$; this uses only that a nonempty Zariski-open subset of a $k$-rational variety has Zariski-dense $k$-points whenever $k$ is infinite: transport the standard fact for $\Aff^N_k$ (a nonempty Zariski-open subset of $\Aff^N_k$ has dense $k$-points over an infinite field, since the complement of finitely many hypersurfaces cannot contain a Zariski-dense set of points of a coordinate line through a generic point) along a birational isomorphism $\calR_\rho\dashrightarrow \Aff^{3n}_k$.  For the dimension of the family of distinct curves, two redundancies must be removed from $\dim\calR_\rho=3n$.  First, the diagonal $PGL_2$ acting on the target does not change the curve traced in ${}^\rho\overline M_{0,n}$; quotienting by it gives $Q\calR_\rho$ of dimension $3n-3$ (\S\ref{sec:unirationality}), the space of \emph{parametrized} M\"obius curves.  Second, precomposition with a reparametrization of the source, $v=(g_i)\mapsto(g_ih)$, leaves the image curve unchanged and only reparametrizes $\gamma_v$.  Distinct curves are therefore parametrized by the further quotient $Q\calR_\rho/\PGL_2^{\mathrm{pre}}$, of dimension $3n-6$ --- the unparametrized M\"obius locus $Q\calR^{\mathrm{un}}$ of Proposition \ref{prop:moebius}, which is the parameter space for them.  Part (c) is stated at the level of the union of the images, where both redundancies are immaterial.

(c) The evaluation $ev_{t_0}$, composed with the quotient map, is dominant by Proposition \ref{prop:evaluation-dominant}, and by (b) the source has dense $k$-points; density of the image points follows.  In detail, let $Z:=\overline{ev_{t_0}(\calR_\rho^\circ(k))}$, the Zariski closure (over $k$) of the image points, and suppose toward a contradiction that $Z\subsetneq{}^\rho\overline M_{0,n}$ is a proper closed $k$-subscheme. Since $ev_{t_0}$ (composed with $\pi_\rho$) is dominant (Proposition \ref{prop:evaluation-dominant}), the nonempty open $k$-subscheme ${}^\rho\overline M_{0,n}\setminus Z$ pulls back to a nonempty open $k$-subscheme $W\subset\calR_\rho^\circ$ where $ev_{t_0}$ is defined. By part (b), $\calR_\rho^\circ(k)$ is Zariski dense in $\calR_\rho^\circ$, hence meets every nonempty open $k$-subscheme; in particular there is a $k$-point $v\in W(k)$ with $ev_{t_0}(v)\notin Z$. But $v\in\calR_\rho^\circ(k)$, so $ev_{t_0}(v)\in ev_{t_0}(\calR_\rho^\circ(k))\subset Z$ by definition of $Z$ --- a contradiction. Hence $Z={}^\rho\overline M_{0,n}$: the image points are Zariski dense.
\end{proof}

\section{Twisting the source: correspondences from a conic}\label{sec:cor}

Everywhere above, the Galois cocycle was valued in $S_n\subset N=PGL_2\times S_n$: the target diagonal $PGL_2$ was left untwisted, so the source $\PP^1$ parametrizing a M\"obius curve stayed split. We now allow the full normalizer $N$, and see what changes when the $PGL_2$-component is nontrivial too.

\subsection{The graph reformulation}

Let $\widehat X:=\PP^1_0\times X=(\PP^1)^{\hat S}$, $\hat S:=\{0\}\sqcup S$. A tuple $f=(A_i)_{i\in S}\in\calR^\circ\cong PGL_2^S$ has a graph
\[
\Gamma_f:=\{(t,f(t)):t\in\PP^1\}\subset\widehat X,
\]
a curve of multidegree $(1,1,\dots,1)$ (degree $1$ over \emph{every} one of the $n+1$ factors, not just the target ones): the projection $\Gamma_f\to\PP^1_0$ is an isomorphism by construction, and $\Gamma_f\to i$-th target factor has degree $\le1$ by definition of $\calR^\circ$. Conversely, an irreducible curve $C\subset\widehat X$ of multidegree $(1,\dots,1)$ has $C\to\PP^1_0$ finite of degree $1$, hence (source and target both smooth proper curves) an isomorphism, so $C=\Gamma_f$ for the induced $f\in\calR^\circ$. Thus
\[
\calR^\circ\ \cong\ \{\text{curves in }\widehat X\text{ of multidegree }(1,\dots,1)\},
\]
 symmetric in all $n+1$ factors.

\subsection{Two actions of $PGL_2$ on $\calR^\circ$}

$PGL_2$ acts on $\calR^\circ=PGL_2^S$ in (at least) two ways, corresponding to two different actions on $\widehat X$:
\begin{itemize}
\item \emph{Post-composition} (used throughout \S\ref{sec:universal}--\ref{sec:cor}): $g\cdot f:=g\circ f$, i.e.\ $A_i\mapsto GA_i$ (left multiplication). On $\widehat X$ this is $g$ acting on the target factors alone, trivially on $\PP^1_0$: $(t,x)\mapsto(t,g\cdot x)$.
\item \emph{Diagonal conjugation}: extend the diagonal $PGL_2$ of Lemma \ref{lem:normalizer} to act on \emph{all} $n+1$ factors of $\widehat X$ simultaneously, $(t,x_1,\dots,x_n)\mapsto(g\cdot t,g\cdot x_1,\dots,g\cdot x_n)$. Reparametrizing $\Gamma_f$ by $t':=g(t)$ shows this sends $\Gamma_f$ to $\Gamma_{g f g^{-1}}$: on $\calR^\circ=PGL_2^S$ it is \emph{conjugation}, $A_i\mapsto GA_iG^{-1}$.
\end{itemize}

\begin{lemma}\label{lem:hatN}
Let $PGL_2^\Delta\subset PGL_2^{\hat S}$ be diagonal and let $\hat N:=N_{\Aut(\widehat X)}(PGL_2^\Delta)\cap\mathrm{Stab}(0)$ (automorphisms normalizing the diagonal and fixing the partition $\{0\}\sqcup S$). Then $\hat N\cong PGL_2\times S_n$, a \emph{direct} product.
\end{lemma}
\begin{proof}
As in Lemma \ref{lem:normalizer}, applied to $\hat S$ in place of $S$ and intersected with the subgroup of $\Sym(\hat S)$ fixing $0$, which is exactly $\Sym(S)$.  The product is direct because $S_n=\Sym(S)$ centralizes $PGL_2^\Delta$: conjugating a constant tuple $(g,\dots,g)$ by a permutation of its (equal) entries returns $(g,\dots,g)$.
\end{proof}

This is the same abstract group $N=PGL_2\times S_n$ as before (indeed Lemma \ref{lem:normalizer}'s $N$ is likewise a direct product, by the identical centralizing observation) --- but embedded in $\Aut(\widehat X)$ differently, acting on $\calR^\circ$ by conjugation$+$permutation rather than left-multiplication$+$permutation.

\subsection{Twisting by $(\alpha,\rho)\in H^1(k,PGL_2)\times H^1(k,S_n)$}

Since $\hat N=PGL_2\times S_n$ is a direct product, $H^1(k,\hat N)=H^1(k,PGL_2)\times H^1(k,S_n)$: a cocycle is exactly an independent pair $(\alpha,\rho)$, $\alpha\in H^1(k,PGL_2)$ giving a conic $D_\alpha$, and $\rho$ a degree-$n$ \'etale algebra $T_\rho$ as before.

The following map is an injection
\[
  H^1(k,PGL_2)\hookrightarrow\Br(k)[2]
\]
whose image is the set of classes of index dividing $2$, i.e.\ the classes represented by a quaternion algebra. So $H^1(k,PGL_2)$ is the set of isomorphism classes of quaternion algebras over $k$, and we write $Q_\alpha$ for the quaternion algebra representing $\alpha$.  This is precisely why $\alpha$ determines a conic $D_\alpha=\SB(Q_\alpha)$. Wherever $\Br(k)[2]$ appears below as the source of an $\alpha$, it should be read as $H^1(k,PGL_2)$.

\begin{lemma}\label{claim:hatXtwist}
$\widehat X_{(\alpha,\rho)}\cong D_\alpha\times X^\alpha_\rho$, where $X^\alpha_\rho:=\Res_{T_\rho/k}(D_\alpha)_{T_\rho}$.
\end{lemma}
\begin{proof}
The diagonal $PGL_2$ twists \emph{every} factor of $\widehat X=(\PP^1)^{\hat S}$ uniformly by $\alpha$, replacing $\PP^1$ by $D_\alpha$ in each: $\widehat X_\alpha\cong(D_\alpha)^{\hat S}=\PP^1_0{}_\alpha\times(D_\alpha)^S$. The index-$0$ factor is untouched by $\rho$ (which only permutes $S$), so it remains $D_\alpha$; further twisting $(D_\alpha)^S$ by $\rho$ is exactly the argument of Lemma \ref{claim:X} with $D_\alpha$ in place of $\PP^1$ throughout, giving $(D_\alpha)^S_\rho\cong\Res_{T_\rho/k}(D_\alpha)_{T_\rho}$.
\end{proof}

 Quotienting $\widehat X$ by the diagonal $PGL_2$ alone (forgetting $S_n$) gives, by the very construction of \S\ref{sec:universal}--\ref{sec:unirationality} with $n+1$ in place of $n$, $\widehat X/\!\!/PGL_2\cong M_{0,n+1}^{\mathrm{GIT}}$ --- the classical fact that forgetting one marked point exhibits $\overline M_{0,n+1}\to\overline M_{0,n}$ as the universal curve, transported to the GIT model. Twisting: ${}^{(\alpha,\rho)}\widehat X/\!\!/PGL_1(Q_\alpha)={}^{(\alpha,\rho)}M_{0,n+1}^{\mathrm{GIT}}$, using Lemma \ref{claim:hatXtwist} and that the diagonal $PGL_2$ itself twists, by conjugation on itself, to the inner form $PGL_1(Q_\alpha)=\Aut(D_\alpha)$.

\subsection{Rationality of the graph space, even when $\alpha\ne0$}

Unlike $D_\alpha$ itself (not $k$-rational when $\alpha\ne0$: it has no $k$-point at all), the space of graphs remains rational.

\begin{lemma}[$PGL_1$ of a quaternion algebra is rational]\label{lem:PGL1rational}
For every $\alpha\in H^1(k,PGL_2)$, with $Q_\alpha$ the quaternion algebra representing it, $PGL_1(Q_\alpha)$ is a $k$-rational variety, of dimension $3$.
\end{lemma}
\begin{proof}
$PGL_1(Q_\alpha)\cong\mathrm{SO}(q_\alpha)$, the special orthogonal group of the $3$-dimensional quadratic form $q_\alpha$ given by the reduced norm restricted to trace-zero elements of $Q_\alpha$ (the conjugation action of $Q_\alpha^\times$ on its trace-zero part is exactly the adjoint representation, identifying $PGL_1(Q_\alpha)$ with $\mathrm{SO}(q_\alpha)$). The Cayley transform $A\mapsto(I-A)(I+A)^{-1}$ gives a birational map, defined purely in terms of the (split or not) quadratic space, from the $3$-dimensional space of skew-symmetric endomorphisms of $q_\alpha$ (an affine space, since skew-symmetric matrices for \emph{any} nondegenerate quadratic form over $k$, char\,${}\ne2$, form a $k$-vector space) to $\mathrm{SO}(q_\alpha)$; this exhibits $\mathrm{SO}(q_\alpha)\cong PGL_1(Q_\alpha)$ as $k$-rational regardless of whether $q_\alpha$ (equivalently $Q_\alpha$) splits.
\end{proof}

\begin{theorem}\label{thm:Rcirc-rational}
For every $(\alpha,\rho)\in H^1(k,PGL_2)\times H^1(k,S_n)$, $\calR^\circ_{(\alpha,\rho)}$ is $k$-rational.
\end{theorem}
\begin{proof}
By the conjugation description of \S9.2, $\calR^\circ_{(\alpha,0)}\cong\Aut(D_\alpha)^S=PGL_1(Q_\alpha)^S$, rational by Lemma \ref{lem:PGL1rational}. Twisting further by $\rho$ gives $\calR^\circ_{(\alpha,\rho)}\cong\Res_{T_\rho/k}\big(PGL_1(Q_\alpha)_{T_\rho}\big)$ (same functorial twisting argument as Lemma \ref{claim:R}, with $PGL_1(Q_\alpha)$ in place of $\PP(\Mat)$): and this is $k$-rational.  We spell the last step out, because the general assertion ``a Weil restriction of a $k$-rational variety with a $k$-point is $k$-rational'' is false as stated: a rational variety with a rational point need not carry a rational affine chart around that point.  Here we have something better and entirely explicit.  The Cayley transform of Lemma \ref{lem:PGL1rational} identifies a dense open subset of $PGL_1(Q_\alpha)$ with a dense open subset of the affine space $\Aff^3_k$ of skew-symmetric endomorphisms of $q_\alpha$, and this chart contains the identity.  Weil restriction preserves open immersions, so $\Res_{T_\rho/k}$ of that chart is a dense open subset of $\calR^\circ_{(\alpha,\rho)}$ and simultaneously a dense open subset of $\Res_{T_\rho/k}\Aff^3_{T_\rho}=\Aff^{3n}_k$.  Hence $\calR^\circ_{(\alpha,\rho)}$ is $k$-rational.
\end{proof}

\begin{remark}
This does \emph{not} restore unirationality of $X^\alpha_\rho$ (or of ${}^{(\alpha,\rho)}M_{0,n+1}^{\mathrm{GIT}}$) the way Theorem \ref{thm:curves} did for $\alpha=0$: that argument evaluated the universal family at a \emph{fixed $k$-rational point} $t_0\in\PP^1(k)$, and $D_\alpha(k)=\varnothing$ for $\alpha\ne0$ leaves no such point to evaluate at.
\end{remark}

\subsection{Correspondences, and their multidegree components}\label{sec:multideg}

\begin{definition}\label{def:CorRd}
For $d_0\ge1$ and $\mathbf d=(d_1,\dots,d_n)\in\mathbb Z_{\ge1}^n$, let
\[
\mathrm{Cor}_R^{(d_0,\mathbf d)}:=\{\text{curves in }\widehat X\text{ of multidegree }(d_0,d_1,\dots,d_n)\}.
\]
$\hat N=PGL_2\times S_n$ acts on the whole Hilbert scheme of curves in $\widehat X$, graded by multidegree, with $S_n$ permuting the components $\mathrm{Cor}_R^{(d_0,\mathbf d)}$ according to how it permutes the entries of $\mathbf d$. For $(\alpha,\rho)$-twisting to act on a \emph{single} component --- rather than shuffling several components of different shape into one another --- restrict to $S_n$-fixed, i.e.\ \emph{constant}, multidegrees $\mathbf d=(d,\dots,d)$, and write $\mathrm{Cor}_R^{(d_0,d)}$ for the corresponding component.  Here and below, $\mathrm{Cor}_R^{(d_0,\mathbf d)}$ denotes the irreducible component containing the smooth irreducible curves of that multidegree; the Hilbert scheme may have further components, which we do not consider. Then $\calR^\circ=\mathrm{Cor}_R^{(1,1)}$, and $\hat N$ preserves each $\mathrm{Cor}_R^{(d_0,d)}$: the diagonal-conjugation action of Lemma \ref{lem:hatN} is defined on all of $\widehat X$ regardless of $C\subset\widehat X$, so it preserves multidegree, and permuting the $S$-factors preserves a constant $\mathbf d$ termwise.
\end{definition}

$D_\alpha$ has index dividing $2$: embedded as a plane conic, its hyperplane
class $\calO_{D_\alpha}(1)$ is defined over $k$, of degree $2$, with
$h^0(D_\alpha,\calO(1))=3$, and a general pencil inside
$H^0(D_\alpha,\calO(1))$ is base-point-free, giving an honest $k$-morphism
$\pi:D_\alpha\to\PP^1$ of degree $2$ --- unconditionally, whether or not
$D_\alpha$ splits.  Suppose $\alpha\ne0$, so that $D_\alpha(k)=\varnothing$.
Then for every $s\in\PP^1(k)$ the fibre $\pi^{-1}(s)$, a subscheme of length
two, cannot be supported at a $k$-point, so it is a single closed point of
degree exactly two, and in particular reduced: $\pi^{-1}(s)=\Spec$ of a
quadratic field extension of $k$, separable since $\mathrm{char}\,k\ne2$.
Thus $\pi$ exhibits $D_\alpha$ as a family of quadratic \'etale algebras over
$\PP^1$, parametrized $k$-rationally by $s$.  (For $\alpha=0$ the same holds
for every $s$ outside the branch locus of $\pi$, the two branch fibres being
non-reduced; only the case $\alpha\ne0$ is used below.)

\begin{proposition}\label{prop:Cor2}
\begin{enumerate}[label=(\roman*)]
\item $\mathrm{Cor}_R^{(2,1)}$ is $k$-rational, and so is its twist by any
$\rho\in H^1(k,S_n)$, i.e.\ the case $\alpha=0$.
\item For every $(\alpha,\rho)\in H^1(k,PGL_2)\times H^1(k,S_n)$ and every
$s_0\in\PP^1(k)$, evaluation along $Z_2:=\pi^{-1}(s_0)$ carries each
$k$-point of $\mathrm{Cor}_R^{(2,1)}{}^{(\alpha,\rho)}$ to a $k$-rational
closed subscheme of degree $\le2$ of $X^\alpha_\rho$.  For $\alpha=0$ these
subschemes sweep out a Zariski-dense family.
\end{enumerate}
\end{proposition}

\begin{proof}
For the first assertion we appeal to Proposition
\ref{prop:XmultidegRational}, which gives $k$-rationality of the space of
multidegree $(2,1,\dots,1)$ maps $\PP^1\to\widehat X$ over any field, by the
explicit binary-form parametrization: after rigidifying three source points,
such a map is a point of an open subset of $\PP^{5}\times(\PP^3)^n$.  By
Proposition \ref{prop:kRatX}(b) this space is birational to
$\mathrm{Cor}_R^{(2,1)}$, since the multidegree has a coordinate equal to
$1$.  The rigidification used there normalises three points of the source
$\PP^1_0$ and is untouched by $S_n$, which permutes only the target factors;
so for $\alpha=0$ it is $\rho$-equivariant, and the $\rho$-twist is
birational to an open subscheme of
$\PP^5\times\Res_{T_\rho/k}(\PP^3_{T_\rho})$, again $k$-rational by Theorem
\ref{thm:R-rational}.  

For the second, fix $s_0\in\PP^1(k)$ and put $Z_2=\pi^{-1}(s_0)$, a
$k$-rational closed subscheme of $D_\alpha$ of degree $2$ by the discussion
preceding Definition \ref{def:CorRd}.  A $k$-point $v$ of
$\mathrm{Cor}_R^{(2,1)}{}^{(\alpha,\rho)}$ is a curve defined over $k$, so
its restriction to $Z_2$ is a $k$-rational closed subscheme of
$X^\alpha_\rho$ of degree $\le2$; composing with the quotient map gives one
on ${}^{(\alpha,\rho)}M_{0,n+1}^{\mathrm{GIT}}$.  For $\alpha=0$ the
$\rho$-twist is $k$-rational by (i), so its $k$-points are Zariski dense and
the associated subschemes sweep out a dense family, exactly as in Theorem
\ref{thm:curves} but one dimension up in cycle degree.
\end{proof}

\subsection{Evaluation over the splitting field}

Proposition \ref{prop:Cor2} is purely $k$-rational, but only produces degree-$2$ cycles, not points. Using the multidegree freedom of Definition \ref{def:CorRd} together with the one field extension that trivializes $\alpha$: since $\alpha\in\Br(k)[2]$ is represented by a quaternion algebra $Q_\alpha=(a,b)_k$, it splits over the quadratic extension $k':=k(\sqrt a)$, i.e.\ $\alpha_{k'}=0$ and $D_\alpha\times_kk'\cong\PP^1_{k'}$.

\begin{theorem}\label{thm:splitfieldeval}
Let $k'/k$ be a quadratic splitting field for $\alpha$, let $d\ge1$, and
let $\rho\in H^1(k,S_n)$.
\begin{enumerate}[label=(\roman*)]
\item For $d_0=1$: base-changing
$\mathrm{Cor}_R^{(1,d)}{}^{(\alpha,\rho)}$ to $k'$ and evaluating at any
$t_0'\in\PP^1(k')=D_\alpha(k')$ gives a dominant $k'$-rational map to
$X^\alpha_\rho\times_kk'$.
\item Already for $d_0=d=1$, i.e.\ for $\calR^\circ_{(\alpha,\rho)}$ itself,
$X^\alpha_\rho\times_kk'$ (equivalently
${}^{(\alpha,\rho)}M_{0,n+1}^{\mathrm{GIT}}\times_kk'$) is $k'$-unirational
and carries a Zariski-dense family of $k'$-rational curves, for the specific
$k'$ splitting $\alpha$.
\end{enumerate}
\end{theorem}
\begin{proof}
Base-changing to $k'$ trivializes the $PGL_2$-component of the twist:
$\widehat X_{(\alpha,\rho)}\times_kk'\cong\widehat
X_{(0,\rho)}\times_kk'=\PP^1_{k'}\times X_{\rho,k'}$, identifying
$\mathrm{Cor}_R^{(1,d)}{}^{(\alpha,\rho)}\times_kk'$ with the
(higher-degree-in-the-target analogue of the) untwisted-in-$\alpha$
construction of \S\ref{sec:model}, base-changed to $k'$.  Because $d_0=1$ the
projection $C\to\PP^1_0$ is an isomorphism for every curve $C$ in the
component, so such a $C$ is the graph of a morphism $\PP^1_{k'}\to X_{\rho,k'}$
of multidegree $\mathbf d$ and evaluation at a fixed point
$t_0'\in\PP^1(k')$ (which exists in abundance, $\PP^1(k')$ being infinite) is
defined.  It is dominant onto each target factor by the identical
linear-system argument as Proposition \ref{prop:evaluation-dominant}
(evaluation of $H^0(\PP^1,\calO(d))$-type data at a point is surjective for
every $d\ge1$, not only $d=1$); composing the factor-wise dominant
evaluations, and then with the (dominant) GIT quotient map, gives (i).  For
$d_0=d=1$ this is exactly Proposition \ref{prop:evaluation-dominant} and
Theorem \ref{thm:curves}, verbatim, carried out over $k'$ instead of $k$, 
which gives (ii).
\end{proof}

\section{Kontsevich moduli of stable maps: Galois descent, and the constructions above as loci}\label{sec:kontsevich}

Everything above used explicit matrices, explicit Weil
restrictions.  We now show that it sits inside the Kontsevich moduli space
of stable maps recalled in \S\ref{sec:prelim}. From here on,
the object of study is $\overline M_{0,m}(X,\beta)$ itself.

\subsection{The case $X=(\PP^1)^n$}\label{ssec:dictionary}

$X=(\PP^1)^n=(\SL_2/B)^n$ is homogeneous for $G=\SL_2^n$, $P=B^n$. Since $\Pic(X)\cong\mathbb Z^n$ freely (one generator per factor), curve classes are multidegrees $\gamma=\mathbf d=(d_1,\dots,d_n)\in\mathbb Z_{\ge0}^n$, and Theorem \ref{thm:KP} gives:

\begin{corollary}\label{cor:KPforX}
$\overline M_{0,r}(X,\mathbf d)$ is irreducible and $\kbar$-rational, for every $\mathbf d\in\mathbb Z_{\ge0}^n$ and $r\ge0$.
\end{corollary}

\begin{proposition}\label{prop:dictionary}
Let $X=(\PP^1)^n$ and $\widehat X=\PP^1_0\times X=(\PP^1)^{n+1}$.
\begin{enumerate}[label=(\alph*)]
\item $M_{0,0}(X,\mathbf1)\cong\calR^\circ/PGL_2^{\mathrm{pre}}$, of dimension
$3n-3$, where $PGL_2^{\mathrm{pre}}$ acts on $\calR^\circ\cong PGL_2^S$ by
precomposition $A_i\mapsto A_iG^{-1}$.
\item $\overline M_{0,1}(X,\mathbf d)$ is the universal curve over
$\overline M_{0,0}(X,\mathbf d)$, and the evaluation map
$ev_1:\overline M_{0,1}(X,\mathbf d)\to X$ is the universal evaluation
$\Psi$ of Construction \ref{con:universal}.
\end{enumerate}
\end{proposition}

\begin{proof}
(a) A stable map of class $\mathbf1$ with smooth irreducible domain is an
$n$-tuple of degree one maps $\PP^1\to\PP^1$, i.e.\ a point of
$\calR^\circ\cong PGL_2^S$, taken up to reparametrization of the source.
That reparametrization is a \emph{third} action of $PGL_2$ on $\calR^\circ$,
namely right multiplication $A_i\mapsto A_iG^{-1}$, complementing the
post-composition (left multiplication, \S\ref{sec:universal}) and
diagonal-conjugation (\S\ref{sec:cor}) actions already in use.  It is simply
transitive on each coordinate, so fixing one index $i_0$ and normalizing
$A_{i_0}=1$ identifies the quotient with $PGL_2^{S\smallsetminus\{i_0\}}$,
of dimension $3n-3$.  This matches the expected dimension
$\dim X+\int_{\mathbf1}c_1(T_X)+r-3=n+2n+0-3=3n-3$.

(b) This is the standard identification of $\overline M_{0,1}(X,\beta)$ with
the universal curve, the forgetful morphism being the universal family.
\end{proof}

\subsection{Markings over a non-closed field}\label{ssec:Emarking}

Over a non-closed field the markings of a stable map need not be
individually $k$-rational.  Let $\mu:C\to X$ be a genus zero stable map over
$k$, so that $C$ is the domain: a proper, connected, at worst nodal
$k$-curve of arithmetic genus $0$, not assumed irreducible and not assumed to
have a $k$-point.  A $k$-object is a \emph{$k$-marking}, a closed subscheme
\[
  \iota:\ \Spec E\ \hookrightarrow\ C^{\mathrm{sm}}
\]
of the smooth locus of $C$, finite \'etale of degree $m$ over $k$, for a
degree $m$ \'etale $k$-algebra $E$.  Over $k^s$ it splits into $m$ disjoint
sections, permuted by $\Gammak$ through the action on $E$.  We write
$\overline M_{0,E}(X,\beta)$ for the corresponding Kontsevich space, as in
\S\ref{ssec:twistedPr} below; the split case $E=k^m$ recovers
$\overline M_{0,m}(X,\beta)$.  This distinction here is that the
rigidification below normalizes three markings to $0,1,\infty$, and that
requires three markings that are $k$-rational \emph{individually}.

\subsection{Galois descent on $\overline M_{0,r}(X,\gamma)$}

The construction $\overline M_{0,r}(-,-)$ is functorial: an isomorphism $\phi:X\xrightarrow{\sim}X'$ carrying $\gamma$ to $\gamma'$ induces $\phi_*:\overline M_{0,r}(X,\gamma)\xrightarrow{\sim}\overline M_{0,r}(X',\gamma')$, $[\mu,p_\bullet]\mapsto[\phi\circ\mu,p_\bullet]$. In particular $N=PGL_2\times S_n$ acts on $\overline M_{0,r}(X,\mathbf d)$ whenever $\mathbf d$ is $S_n$-fixed (constant, $\mathbf d=(d,\dots,d)$, so that $N$ preserves the class), by post-composing stable maps with $N\subset\Aut(X)$.

\begin{definition}
For $\rho\in H^1(k,N)$ (in particular for $\rho=(\alpha,\rho_0)\in H^1(k,PGL_2)\times H^1(k,S_n)$ as in \S\ref{sec:cor}), define
\[
{}^\rho\overline M_{0,r}(X,\mathbf d):=\overline M_{0,r}({}^\rho X,\mathbf d).
\]
\end{definition}

This is well-posed by the same functoriality: twisting the target by $\rho$ and forming the Kontsevich space commute, because the Kontsevich space construction is defined over $\mathrm{Spec}(k)$ functorially in the pair $(X,\mathbf d)$, so Galois descent along $\rho$ applies uniformly to $X$ and to $\overline M_{0,r}(X,\mathbf d)$ simultaneously. 

\subsection{Galois descent on $\overline M_{0,r}(\overline M_{0,n},\gamma')$}

The diagonal action used to form $X/\!\!/PGL_2=M_{0,n}^{\mathrm{GIT}}$ gives, dually, a natural curve class on the quotient: the pushforward under $\pi:X^{ss}\to M_{0,n}^{\mathrm{GIT}}$ of the class $\mathbf1$ defines $\gamma':=\pi_*\mathbf1\in N_1(M_{0,n}^{\mathrm{GIT}})$ --- the class swept out by $Q\calR$ (\S\ref{sec:unirationality}). Here $-K\cdot\gamma'=2n$ by Proposition \ref{prop:moebius}. Since $S_n=\Aut_k(M_{0,n}^{\mathrm{GIT}})$ \cite{BM} fixes $\gamma'$, we may form $\overline M_{0,r}(M_{0,n}^{\mathrm{GIT}},\gamma')$, and for $\rho_0\in H^1(k,S_n)$:
\[
{}^{\rho_0}\overline M_{0,r}(M_{0,n}^{\mathrm{GIT}},\gamma'):=\overline M_{0,r}\big({}^{\rho_0}M_{0,n}^{\mathrm{GIT}},\gamma'\big),
\]
by the same functoriality as above, now applied to the target $M_{0,n}^{\mathrm{GIT}}$ (already established, \S\ref{ssec:forms}, to be classified by exactly this $H^1(k,S_n)$).

\section{Proof of Theorem 1: the split targets $\PP^r$ and $(\PP^1)^n$}
\label{sec:Pr}

This section proves Theorem \ref{thmA}.  The two targets are treated in turn:
$\PP^r$ in \S\ref{ssec:splitPr}, by the binary-form model of Coskun and
Pandharipande, which we observe holds over the ground field; and
$(\PP^1)^n$ in \S\ref{ssec:splitX}, for an arbitrary \'etale marking algebra
$E$, by a direct convexity argument.

Let $k$ be a field, $r\ge1$, $d\ge1$, $n\ge0$.  Recall
\begin{equation}\label{eq:dimPr}
  \dim\overline M_{0,n}(\PP^r,d)=(r+1)(d+1)+n-4 .
\end{equation}
Unlike $\overline M_{0,n}$, the target $\PP^r$ \emph{is} convex.  Hence the
Kontsevich \emph{stack} $\overline M_{0,n}(\PP^r,d)$ is smooth of this
dimension, with boundary a normal crossings divisor, while its coarse moduli
space is an irreducible normal projective variety with quotient
singularities; the automorphism-free locus is a \emph{fine} moduli space
carrying a universal family \cite[Thm.~1.10]{Cos}.  
Over the
automorphism-free locus, a $k$-point really is a stable map defined over
$k$.

\subsection{The target $\PP^r$: a birational model over any field}\label{ssec:splitPr}

Over $\mathbb{C}$, rationality of $\overline M_{0,r}(X,\gamma)$ for homogeneous $X$
is by Kim--Pandharipande (Theorem \ref{thm:KP} above).  The
following is the birational model used by Pandharipande \cite{Pan} and Coskun
\cite{Cos} to compute $\Pic$; we only need its geometric content. 

\begin{proposition}[after {\cite[Claim 2.2]{Cos}}]\label{prop:claim22}
Let $n\ge3$.  Put
\[
  W\;=\;(\PP^1)^{n-3}\smallsetminus\big(\text{diagonals}\cup\{x_i\in\{0,1,\infty\}\}\big),
  \qquad
  U\;\subset\;\PP\Big(\textstyle\bigoplus_{i=0}^{r}H^0(\PP^1,\mathcal{O}(d))\Big)
  \;=\;\PP^{(r+1)(d+1)-1}
\]
with $U$ the locus of base-point-free $(r+1)$-tuples of binary forms of
degree $d$ (its complement has codimension $\ge2$).  Then, over any
field $k$, there is an isomorphism onto an open subscheme
\[
  W\times U\;\xrightarrow{\ \sim\ }\;\overline M_{0,n}(\PP^r,d)\smallsetminus\partial ,
\]
where $\partial$ is the boundary.  Consequently:
\begin{enumerate}[label=(\roman*)]
\item $\overline M_{0,n}(\PP^r,d)$ is $k$-rational for every $n\ge3$;
\item for $n=3$ it is birational to $\PP^{(r+1)(d+1)-1}_k$.
\end{enumerate}
\end{proposition}

\begin{proof}
A stable map with irreducible domain and $n\ge3$ distinct markings has
no automorphisms, and there is a \emph{unique} isomorphism of the domain
with $\PP^1$ carrying $(p_1,p_2,p_3)$ to $(0,1,\infty)$.  This
normalisation is defined over $k$ because $0,1,\infty\in\PP^1(k)$ and
the three sections are individually $k$-rational; hence it is an
isomorphism of functors over $k$.  After it,
the datum is precisely a point of $W$ (the remaining markings) together
with a base-point-free $(r+1)$-tuple of binary $d$-forms up to common
scalar, i.e.\ a point of $U$.  Rationality follows since $W$ is open in
$(\PP^1)^{n-3}$ and $U$ is open in a projective space.  The dimension
$(n-3)+(r+1)(d+1)-1$ agrees with \eqref{eq:dimPr}.
\end{proof}

We use the following well-known lemma over a non-closed field.

\begin{lemma}{\cite{BK85,Dol87}}; see also {\cite{CTS07}}]\label{lem:noname}
Let $G$ be a linear algebraic group over a field $k$, let $X$ be a
geometrically integral $k$-variety with a $G$-action, and let
$E\to X$ be a $G$-equivariant vector bundle of rank $m$.  Assume that the
action on $X$ is generically free, in the sense that there are a dense
$G$-stable open subscheme $U\subset X$ and a $G$-torsor $\pi\colon U\to Y$
over a $k$-variety $Y$.  Then
\[
  E/G\;\sim_k\;(X/G)\times\Aff^m_k ,
  \qquad\text{equivalently}\qquad
  k(E)^G\;=\;k(X)^G(t_1,\dots,t_m).
\]
\end{lemma}

\noindent
If $\mathrm{char}\,k=0$, the hypothesis holds as soon as the stabiliser of a
general geometric point of $X$ is trivial, by Rosenlicht's theorem on
generic quotients; see e.g.\ \cite{CTS07}.

\begin{proof}
Replace $X$ by $U$ and $E$ by $E|_U$.  Since $\pi$ is a $G$-torsor, it is
faithfully flat and $k(Y)=k(U)^G=k(X)^G$, so $Y$ is a model of $X/G$.  A
$G$-linearisation of $E$ is precisely a descent datum for $E$ along $\pi$, so
by faithfully flat descent of quasi-coherent sheaves \cite[\S6.1, Thm.~4]{BLR}
there is a quasi-coherent sheaf $\mathcal F$ on $Y$ with $E\cong\pi^*\mathcal F$
compatibly with the $G$-actions; $\mathcal F$ is locally free of rank $m$,
since this property descends along faithfully flat morphisms.  Writing
$\mathbb V(\mathcal F)\to Y$ for the associated vector bundle, we get
$E\cong U\times_Y\mathbb V(\mathcal F)$ with $G$ acting on the first factor
only, so $E\to\mathbb V(\mathcal F)$ is again a $G$-torsor (base change of
$\pi$) and $E/G\sim\mathbb V(\mathcal F)$.  Finally $\mathcal F$ is Zariski
locally free (this is Hilbert's Theorem~90, i.e.\ $\GL_m$ is special), so over a
dense open $Y_0\subset Y$ we have $\mathbb V(\mathcal F)|_{Y_0}\cong
Y_0\times\Aff^m_k$.  Hence $E/G\sim_k Y\times\Aff^m_k\sim_k(X/G)\times\Aff^m_k$.
\end{proof}

The next reduction is the one underlying Kim--Pandharipande's
rationality assertion: as recalled after Theorem \ref{thm:KP}, they deduce
rationality from a torus degeneration together with the Katsylo--Bogomolov
rationality theorems for quotients of $\SL_2$-representations
\cite{Kat,Bog}, which we use below.

We first fix the notation.  Let $W=k^2$ carry the standard $\GL_2$-action, so
that $\PP^1=\PP(W)$, and let $\Sym^d:=\Sym^d(W^\vee)$ be the
$(d+1)$-dimensional space of binary forms of degree $d$ in the coordinates
$x,y$ on $W$, with $\GL_2$ acting by $(g\cdot f)(w)=f(g^{-1}w)$.  Put
\[
V\;:=\;\{(f_0,\dots,f_r):f_i\in\Sym^d\},
\]
 on which $\GL_2$
acts diagonally and $\mathbb{G}_m$ acts by scaling the tuple; thus
$\dim V=(r+1)(d+1)$.  Let $V^{\mathrm{bpf}}\subset \PP(V)$ be the open locus of
tuples whose members have no common zero on $\PP^1_{\bar k}$.  Such a tuple
determines a morphism of degree $d$
\[
\varphi_{(f_0,\dots,f_r)}\colon\PP^1\longrightarrow\PP^r,
\qquad w\longmapsto[f_0(w):\dots:f_r(w)].
\]
Then $V^{\mathrm{bpf}}$ is the space $U$ of Proposition
\ref{prop:claim22}.  Finally let
\[
H\;=\;(\GL_2\times\mathbb{G}_m)/Z ,
\qquad
Z=\{(\mu\cdot\mathrm{id}_W,\,\mu^{d}):\mu\in\mathbb{G}_m\}\cong\mathbb{G}_m ,
\]
act on $(\PP^1)^n\times V$, with $\GL_2$ acting on the $n$ source points and
on $V$ as above and $\mathbb{G}_m$ by scaling $V$; the subgroup $Z$ is
exactly the subgroup acting trivially, since a scalar $\mu$ acts trivially on
$\PP^1$ and by $\mu^{-d}$ on $\Sym^d$.

\begin{theorem}\label{thm:noname}
Assume $\mathrm{char}\,k=0$, $r\ge1$ and $d\ge3$.  Write
\[
  M_{0,n|d}\;:=\;\big((\PP^1)^n\times\PP^d\big)/\PGL_2 ,
\]
the moduli of $n$ ordered points together with one unordered degree $d$
divisor on $\PP^1$, where $\PP^d=\PP(\Sym^d)$.  Then there is a
birational equivalence over $k$
\[
  \overline M_{0,n}(\PP^r,d)\;\sim_k\;M_{0,n|d}\times\Aff^{\,r(d+1)}_k .
\]
\end{theorem}

\begin{proof}
 The interior of
$\overline M_{0,n}(\PP^r,d)$ is the quotient of
$(\PP^1)^n_{\mathrm{distinct}}\times V^{\mathrm{bpf}}$ by $H$, and the
$H$-action there is generically free.  All quotients in this proof are
birational quotients, formed on a dense open locus on which the action is
free and a geometric quotient exists. Split $k^{r+1}=k\oplus k^{r}$;
this is an $H$-stable decomposition
$V=A\oplus B$ with $A=\Sym^d$ the first component $f_0$ and
$B$ parametrises the remaining $r$ components $(f_1,\dots,f_r)$, so
$(\PP^1)^n\times V\to(\PP^1)^n\times A$ is an $H$-equivariant vector
bundle of rank $r(d+1)$.  For $d\ge3$ a general binary form of degree
$d$ has trivial stabiliser in $\PGL_2$, so $H$ already acts generically
freely on $(\PP^1)^n\times A$.  Lemma \ref{lem:noname}, applied to the $H$-equivariant vector bundle
$(\PP^1)^n\times V\to(\PP^1)^n\times A$ with fibre $B$, gives
\[
  \big((\PP^1)^n\times V\big)/H\;\sim\;
  \Big(\big((\PP^1)^n\times A\big)/H\Big)\times\Aff^{r(d+1)} .
\]
Finally $\big((\PP^1)^n\times A\big)/H=\big((\PP^1)^n\times\PP(A)\big)/\PGL_2
=M_{0,n|d}$, the scaling $\mathbb{G}_m$ having produced $\PP(A)=\PP^d$.
Dimensions check: $(n+d-3)+r(d+1)=(r+1)(d+1)+n-4$.
\end{proof}

\begin{remark}\label{rem:KP}
Over $k=\mathbb C$, the birational equivalence of Theorem \ref{thm:noname} is
implicit in the proof of \cite[Thm.~3]{KP}. 
\end{remark}

\begin{corollary}\label{cor:allcases}
With $\mathrm{char}\,k=0$, $r\ge1$, $d\ge3$:
\begin{enumerate}
\item[(i)] $\overline M_{0,n}(\PP^r,d)$ is $k$-rational for $n\ge3$;
\item[(ii)] $\overline M_{0,1}(\PP^r,d)$ and $\overline M_{0,2}(\PP^r,d)$ are
$k$-rational;
\item[(iii)] if $M_{0,0|d}=\PP^d/\PGL_2$, the moduli of $d$ unordered points
on $\PP^1$ (equivalently of binary forms of degree $d$, equivalently
$\overline M_{0,d}/S_d$), is $k$-rational, then $\overline M_{0,0}(\PP^r,d)$
is $k$-rational; conversely, $k$-rationality of $\overline M_{0,0}(\PP^r,d)$
gives only that $M_{0,0|d}$ is stably $k$-rational.
\end{enumerate}
\end{corollary}

\begin{proof}
(i) is Proposition~\ref{prop:claim22}; equivalently, $M_{0,n|d}$ is rational
because three of the ordered points normalise $\PGL_2$.

(ii)  By Theorem \ref{thm:noname} it is enough to treat $M_{0,1|d}$ and
$M_{0,2|d}$.  The first is the space of degree $d$ divisors on $\Aff^1$
modulo the affine group; killing the subleading coefficient and quotienting
by $\mathbb{G}_m$ identifies it with the weighted projective space
$\PP(2,3,\dots,d)$, which is rational.  For $M_{0,2|d}$, normalise the two
marked points to $0,\infty$; the residual subgroup of $\PGL_2$ is the torus
$\mathbb G_m$, $t\mapsto\lambda t$.  A degree $d$ divisor is a binary form
$F=\sum_{j=0}^d a_jx^jy^{d-j}$ up to scale, and $\lambda$ acts by
$a_j\mapsto\lambda^ja_j$.  On the open locus $a_0\ne0\ne a_1$, first use the
overall scalar to set $a_0=1$, then use $\lambda$ --- uniquely determined ---
to set $a_1=1$; the remaining coordinates $(a_2,\dots,a_d)$ are free.  This
exhibits a rational affine chart $M_{0,2|d}\sim\Aff^{d-1}$, of the right
dimension $2+d-3=d-1$.  Hence $M_{0,2|d}$ is $k$-rational.

(iii)  Both implications are read off from the birational equivalence
$\overline M_{0,0}(\PP^r,d)\sim_kM_{0,0|d}\times\Aff^{\,r(d+1)}$ of Theorem
\ref{thm:noname}.  
\end{proof}

\begin{remark}
The hypothesis of (iii) holds for all $d$ over an algebraically closed field
of characteristic zero, by Katsylo and Bogomolov \cite{Kat,Bog}; over a
general $k$ it is open in general.
\end{remark}

\subsection{The target $(\PP^1)^n$, with arbitrary markings}\label{ssec:splitX}

The target is now $X=(\PP^1)^n$, and the markings are no longer assumed
individually $k$-rational: they form a closed subscheme
$\Spec E\hookrightarrow C^{\mathrm{sm}}$ as in \S\ref{ssec:Emarking}, of which
the split case $E=k^m$ treated for $\PP^r$ above is one instance.  The
rigidification used there is available only when three of the markings are
$k$-points, and the proposition isolates the two circumstances in which the
source $\PGL_2$ can be rigidified over $k$ regardless: three $k$-rational
markings, or a coordinate of degree one.

\begin{proposition}\label{prop:kRatX}
Let $X=(\PP^1)^n$ and $\widehat X=\PP^1_0\times X=(\PP^1)^{n+1}$, over an
arbitrary field $k$.  Let $\mathbf d=(d_1,\dots,d_n)\in\mathbb Z_{\ge1}^n$,
let $d_0\ge1$, and let $E$ be a degree $m$ \'etale $k$-algebra.
\begin{enumerate}[label=(\alph*)]
\item $\overline M_{0,E}(X,\mathbf d)$ is irreducible of dimension
$n+2\sum_i d_i+m-3$, and is $k$-unirational for every $E$.  It is
$k$-rational in either of the following two cases:
\begin{enumerate}[label=(\roman*)]
\item $E$ has at least three factors isomorphic to $k$, i.e.\ the marking
contains three $k$-rational points --- in particular whenever $E=k^m$ with
$m\ge3$;
\item $d_{i_0}=1$ for some index $i_0$, for arbitrary $E$ and arbitrary
$m\ge0$.
\end{enumerate}
\item If at least one of $d_0,d_1,\dots,d_n$ equals $1$, then
$\mathrm{Cor}_R^{(d_0,\mathbf d)}$ is $k$-rational.  In particular
$\mathrm{Cor}_R^{(d_0,1)}$ is $k$-rational for every $d_0\ge1$, which covers
$\calR^\circ=\mathrm{Cor}_R^{(1,1)}$ and $\mathrm{Cor}_R^{(2,1)}$ of
\S\ref{sec:cor}.
\end{enumerate}
\end{proposition}

\begin{proof}
\emph{Step 1: convexity, and the interior is dense.}
We first check that $X$ is convex over any field, i.e.\ that
$H^1(C,\mu^*T_X)=0$ for every genus zero stable map $\mu:C\to X$.  Indeed
$T_X=\bigoplus_{i}\mathrm{pr}_i^*T_{\PP^1}=\bigoplus_i\mathrm{pr}_i^*\calO(2)$,
so $\mu^*T_X=\bigoplus_i f_i^*\calO(2)$ where $f_i=\mathrm{pr}_i\circ\mu$;
on each irreducible component of $C$ this has degree $2\deg(f_i|_{\text{component}})\ge0$,
and a line bundle of nonnegative degree on every component of a connected
nodal curve of arithmetic genus $0$ has vanishing $H^1$.  Hence
$\overline M_{0,m}(X,\mathbf d)$ is a smooth Deligne--Mumford stack of pure
dimension equal to the expected one,
\[
  \dim X+\big(-K_X\cdot\mathbf d\big)+m-3=n+2\sum_i d_i+m-3 ,
\]
since $-K_X=\sum_i 2H_i$.

The boundary $\partial$, the locus of maps with reducible domain, is an
effective Cartier divisor: locally it is cut out by the node-smoothing
parameter in the versal deformation.  A Cartier divisor on a smooth stack
contains no irreducible component of that stack, so every irreducible
component of $\overline M_{0,m}(X,\mathbf d)$ meets the interior
$M_{0,m}(X,\mathbf d)$ of maps with smooth domain.  We check below that this
interior is irreducible; granting that, $\overline M_{0,m}(X,\mathbf d)$ is
irreducible with dense interior, and it suffices to exhibit a $k$-rational
birational model of the interior.  Note that this argument is
characteristic-free and does not invoke Theorem \ref{thm:KP}.

\emph{Step 2: the interior, explicitly.}
A stable map with smooth domain is a morphism $\mu:\PP^1\to X$ together with
$m$ distinct marked points, taken up to the precomposition action of
$\PGL_2$ on the source.  Giving $\mu$ is giving the $n$ coordinates
$f_i:\PP^1\to\PP^1$, and a morphism $\PP^1\to\PP^1$ of degree $d$ is a pair
$(a:b)$ of binary forms of degree $d$ with no common zero, up to a common
scalar.  Thus
\[
  \Mor_d(\PP^1,\PP^1)\;=\;\PP\big(H^0(\calO(d))^{\oplus2}\big)\smallsetminus
  \{\mathrm{Res}(a,b)=0\}\;\subset\;\PP^{2d+1},
\]
an open subvariety of a projective space, over $k$, dense
and irreducible.  Writing
$\Mor_{\mathbf d}:=\prod_{i=1}^n\Mor_{d_i}(\PP^1,\PP^1)$, an irreducible
$k$-rational variety of dimension $\sum_i(2d_i+1)=2\sum_id_i+n$, we get
\begin{equation}\label{eq:interiorX}
  M_{0,E}(X,\mathbf d)\;=\;\Big(\mathrm{Emb}_E\times\Mor_{\mathbf d}\Big)\Big/\PGL_2 ,
\end{equation}
where
\[
  \mathrm{Emb}_E:=\{\text{closed immersions }\Spec E\hookrightarrow\PP^1\}
\]
is the open subscheme of $\Res_{E/k}(\PP^1_E)$ where the $m$ geometric points
are distinct, and $\PGL_2$ acts diagonally, by
$g\cdot(\iota,(f_i))=(g\circ\iota,(f_i\circ g^{-1}))$.  Note that
$\mathrm{Emb}_E$ is $k$-rational, being open and dense in
$\Res_{E/k}(\PP^1_E)$, which is birational to
$\Res_{E/k}(\Aff^1_E)=\Aff^m_k$; for $E=k^m$ it is
$(\PP^1)^m\smallsetminus\Delta$.  The source of \eqref{eq:interiorX} is
irreducible, so the interior is irreducible, as claimed in Step 1.  The
action is free whenever $m\ge3$ (three distinct geometric points have trivial
stabiliser) and also whenever some $d_{i_0}=1$ (see Step 4), so in those
cases the quotient is a geometric quotient of varieties.

\emph{Step 3: case (i), three $k$-rational markings.}
Write $E=k^3\times E'$, so that the marking contains three $k$-points
$p_1,p_2,p_3$ and a residual $E'$-marking.  Given $(p_1,p_2,p_3)$ distinct there is a \emph{unique}
$g\in\PGL_2$ with $(g(p_1),g(p_2),g(p_3))=(0,1,\infty)$, and it depends
algebraically on $(p_1,p_2,p_3)$; since $0,1,\infty\in\PP^1(k)$, this is
defined over $k$.  Hence the locally closed subvariety
\[
  \Sigma:=\{(p_\bullet,(f_i)):p_1=0,\ p_2=1,\ p_3=\infty\}
\]
is a $k$-rational slice meeting every $\PGL_2$-orbit exactly once, and
\[
  M_{0,E}(X,\mathbf d)\;\cong\;\Sigma\;\cong\;W\times\Mor_{\mathbf d},
  \qquad W:=\{\iota'\in\mathrm{Emb}_{E'}:\ \mathrm{im}(\iota')\cap\{0,1,\infty\}=\emptyset\},
\]
an open subvariety of $\Res_{E'/k}(\PP^1_{E'})\times\prod_i\PP^{2d_i+1}$, hence
$k$-rational of dimension $(m-3)+2\sum_id_i+n$, matching Step 1.  For
$E=k^m$ this is the familiar
$W=\{(p_4,\dots,p_m)\ \text{distinct},\ \notin\{0,1,\infty\}\}$.  

\emph{Step 4: case (ii), a degree one coordinate.}
Now the markings no longer rigidify the source, so we use the map $f_{i_0}$ instead.  The index $i_0$ is an index of the \emph{target}
$(\PP^1)^n$, which is split, so it is $k$-rational whatever the marking
algebra.  A degree one morphism $\PP^1\to\PP^1$ is an element of
$\PGL_2$, so $\Mor_1(\PP^1,\PP^1)=\PGL_2$ and the precomposition action
$f_{i_0}\mapsto f_{i_0}\circ g^{-1}$ is right translation, which is simply
transitive.  Therefore the locally closed subvariety
\[
  \Sigma':=\{(p_\bullet,(f_i)):f_{i_0}=\mathrm{id}_{\PP^1}\}
\]
meets every orbit exactly once, is defined over $k$ (the index $i_0$ is), and
\[
  M_{0,E}(X,\mathbf d)\;\cong\;\Sigma'\;\cong\;
  \mathrm{Emb}_E\times\prod_{i\ne i_0}\Mor_{d_i}(\PP^1,\PP^1),
\]
$k$-rational because $\mathrm{Emb}_E$ is, of
dimension $m+\sum_{i\ne i_0}(2d_i+1)=m+2\sum_id_i+n-3$, matching Step 1.
This also shows the action is free on the locus in question, completing Step
2.

\emph{Step 4$'$: arbitrary $E$ and arbitrary $\mathbf d$.}
Here neither the markings nor a coordinate rigidify the source, and we
descend from a space with more markings --- three \emph{split} ones, which is
exactly what Step 3 needs.  Adjoining three $k$-rational markings to the
$E$-marking and then forgetting them, with stabilization, gives a morphism
\[
  \pi:\ \overline M_{0,\,E\times k^3}(X,\mathbf d)\longrightarrow\overline M_{0,E}(X,\mathbf d),
\]
defined over $k$, since the forgotten markings are the three split ones and
stabilization is functorial.  It is dominant, indeed surjective: given a
stable map with an $E$-marking, its domain has infinitely many points over
$\kbar$, so three further markings may be added in general position, and the
resulting map is again stable because $\mathbf d$ is positive on every
component of the domain that would otherwise destabilize.  The algebra
$E\times k^3$ has three factors isomorphic to $k$, so the source is
$k$-rational by Step 3, hence $k$-unirational; a variety dominated by a
$k$-unirational variety is $k$-unirational, so
$\overline M_{0,E}(X,\mathbf d)$ is $k$-unirational for every $E$.

Part (a) follows.

\emph{Step 5: part (b).}
Apply (a) to the target $\widehat X=(\PP^1)^{n+1}$, the multidegree
$(d_0,\mathbf d)$ and the empty marking $E=k^0$: by hypothesis one of the $n+1$ degrees is $1$,
so Step 4 applies and $\overline M_{0,0}(\widehat X,(d_0,\mathbf d))$ is
irreducible and $k$-rational.

It remains to compare stable maps with curves.  Let $j$ be an index with
degree $1$.  If $\mu:\PP^1\to\widehat X$ has multidegree $(d_0,\mathbf d)$
then $\mathrm{pr}_j\circ\mu$ has degree $1$, hence is an isomorphism, so
$\mu$ is a closed immersion onto its image and in particular birational onto
it.  Consequently the assignment $\mu\mapsto\mu(\PP^1)$ identifies the
interior $M_{0,0}(\widehat X,(d_0,\mathbf d))$ with the locus of
\emph{irreducible} curves in $\widehat X$ of multidegree
$(d_0,\mathbf d)$: the map is injective because a curve determines its
normalization and $\mu$ is the normalization map, and surjective because any
irreducible curve of this multidegree admits a degree one projection and is
therefore rational and smooth, so equals $\mu(\PP^1)$ for the corresponding
$\mu$.  Both sides are open dense in their respective compactifications,
$\overline M_{0,0}(\widehat X,(d_0,\mathbf d))$ and
$\mathrm{Cor}_R^{(d_0,\mathbf d)}$, the former by Step 1.  Hence
\[
  \mathrm{Cor}_R^{(d_0,\mathbf d)}\ \sim_k\ \overline M_{0,0}(\widehat X,(d_0,\mathbf d)),
\]
and $\mathrm{Cor}_R^{(d_0,\mathbf d)}$ is $k$-rational.  For
$\mathbf d=(1,\dots,1)$ the hypothesis holds for every $d_0\ge1$, giving the
last assertion; at $(d_0,\mathbf d)=(1,\mathbf1)$ one recovers
$\calR^\circ\cong\PGL_2^S$ directly from Step 4.
\end{proof}

\begin{proposition}\label{prop:XmultidegRational}
For every $m\ge3$, every multidegree $\tilde d=(d_1,\dots,d_n)$ with
$d_i\ge1$, and every field $k$, the space
$\overline M_{0,m}\big((\PP^1)^n,\tilde d\big)$ is $k$-rational.
\end{proposition}

\begin{proof}
Exactly as in Proposition \ref{prop:claim22}: normalize $(p_1,p_2,p_3)$ to
$(0,1,\infty)$, which is defined over $k$.  What remains is a point of
$W\subset(\PP^1)^{m-3}$ together with an $n$-tuple of degree $d_i$ morphisms
$\PP^1\to\PP^1$, i.e.\ a point of
$\Mor_{d_i}(\PP^1,\PP^1)$, which is the open locus of coprime pairs of
binary forms of degree $d_i$ inside
$\PP\big(H^0(\calO(d_i))^{\oplus2}\big)=\PP^{2d_i+1}$.  Hence the interior is
open in $(\PP^1)^{m-3}\times\prod_i\PP^{2d_i+1}$, of dimension
$(m-3)+\sum_i(2d_i+1)=m-3+2\sum_id_i+n$, as required.
\end{proof}

\begin{proof}[Proof of Theorem \ref{thmA}]
Part (a) is Proposition \ref{prop:claim22} for the target $\PP^r$ and
Proposition \ref{prop:kRatX}(a) for the target $(\PP^1)^n$; the supplementary
case $m\le2$ with some coordinate degree equal to $1$ is Proposition
\ref{prop:kRatX}(a)(ii).  Part (b) is Theorem \ref{thm:noname}, and the two
implications recorded after it --- rationality of $\overline M_{0,0}(\PP^r,d)$
from that of the moduli of binary forms, and stable rationality of the latter
from the former --- are Corollary \ref{cor:allcases}(iii).
\end{proof}

\section{Proof of Theorem 2: twisted targets and the Brauer obstruction}
\label{ssec:twistedPr}

There are two independent ways to twist, which we look at below.  Let $A$ be a central simple $k$-algebra of degree
$r+1$ and $X=\SB(A)$ the associated Severi--Brauer variety; let $E$ be a
degree $n$ \'etale $k$-algebra, used to twist the labelling of the
markings.  The curve class $d[\text{line}]$ is Galois invariant, so the
Kontsevich space
\[
  \overline M_{0,E}(X,d)
\]
is defined over $k$ and is a $k$-form of $\overline M_{0,n}(\PP^r,d)$; it is
split by any field splitting both $A$ and $E$.  Since
$\overline M_{0,n}(\PP^r,d)$ is $k$-rational for $n\ge1$ by
Corollary~\ref{cor:allcases}, these are forms of a rational variety, and
the question is whether they are rational, unirational, or even have a
$k$-point in the automorphism-free locus.

\begin{theorem}\label{thm:amitsur}
Suppose $\overline M_{0,E}(X,d)$ has a $k$-point lying in the
automorphism-free locus, corresponding to a stable map $f:C\to X$ over
$k$ with $E$-marking $\mu:\mathrm{Spec}\,E\hookrightarrow C$.  Let
$[C]\in\Br(k)[2]$ be the class of the genus zero curve $C$.  Then
\begin{equation}\label{eq:AdC}
  [A]\;=\;d\cdot[C]\qquad\text{in }\Br(k),
\end{equation}
and $[C]$ is split by every factor of $E$.  Consequently:
\begin{enumerate}
\item[(i)] if $d$ is even, then $A$ is split;
\item[(ii)] if some factor of $E$ has odd degree over $k$ --- in
particular if $n$ is odd --- then $C\cong\PP^1_k$ and $A$ is split;
\item[(iii)] if $\ind(A)>2$, or if $\per(A)>2$, then
$\overline M_{0,E}(X,d)$ has \emph{no} $k$-point in the automorphism-free
locus, for any $n$, any $E$ and any $d$;
\item[(iv)] if $A$ is not split, then $d$ is odd, $A\cong M_m(Q)$ for a
quaternion algebra $Q$ with $[Q]=[C]$, and, writing $E=\prod_iL_i$ as a
product of finite separable field extensions of $k$, every $L_i/k$ has even
degree and splits $Q$; in particular $n$ is even.
\end{enumerate}
\end{theorem}

\begin{proof}
Let $Y$ be a smooth proper geometrically integral $k$-variety and
$\delta_Y:\Pic(Y_{\bar k})^{\Gamma}\to\Br(k)$ the boundary map of the
Hochschild--Serre sequence
$0\to\Pic(Y)\to\Pic(Y_{\bar k})^{\Gamma}\to\Br(k)\to\Br(Y)$.  It is
natural in $Y$.  By Amitsur, $\delta_X(\mathcal{O}(1))=[A]$; for a genus
zero curve, $\delta_C(\mathcal{O}(1))=[C]$, the class of the
corresponding quaternion algebra.  Since $f$ has degree $d$,
$f^*\mathcal{O}_X(1)=\mathcal{O}_{C_{\bar k}}(d)$, and naturality gives
$[A]=\delta_C(\mathcal{O}(d))=d[C]$, which is \eqref{eq:AdC}.  The
marking $\mu$ exhibits closed subschemes of $C$ with residue algebras
the factors of $E$; a genus zero curve acquires a rational point over
any field splitting its class, and conversely a closed point of degree
$m$ forces $\ind[C]\mid m$.  Since $[C]$ is $2$-torsion, an odd degree
factor forces $[C]=0$, i.e.\ $C\cong\PP^1_k$; then \eqref{eq:AdC} gives
$[A]=0$.  Statement (i) is \eqref{eq:AdC} with $d$ even.  For (iii),
\eqref{eq:AdC} forces $[A]\in\Br(k)[2]$ with $\ind[A]\le2$.  For (iv), suppose $A$ is not split.  By (i), $d$ is odd, so
$d[C]=[C]$ since $[C]$ is $2$-torsion, and \eqref{eq:AdC} reads $[A]=[C]$;
thus $[A]$ is a nonzero $2$-torsion class, $A\cong M_m(Q)$ for the quaternion
algebra $Q$ with $[Q]=[C]$, and $\ind[C]=2$.  Writing $E=\prod_iL_i$ with
each $L_i/k$ a finite separable field extension, the marking gives a closed
point of $C$ of degree $[L_i:k]$, so $2=\ind[C]$ divides $[L_i:k]$; and an
$L_i$-point of $C=\SB(Q)$ splits $Q$.  Finally $n=\sum_i[L_i:k]$ is a sum of
even numbers.
\end{proof}

\begin{corollary}\label{cor:splitmarkings}
For $n\ge1$ and $E=k^n$ split, the automorphism-free locus of
$\overline M_{0,n}(\SB(A),d)$ has a $k$-point if and only if $A$ is split.
\end{corollary}

\begin{proof}
A $k$-point of the automorphism-free locus is an actual stable map $f:C\to X$
over $k$ with its markings, so a $k$-rational marking is a $k$-point of $C$;
hence $C\cong\PP^1_k$ and $f(p_1)$ is a $k$-point of $X=\SB(A)$, and $A$ is
split.  Conversely, for $A$ split the space is $\overline M_{0,n}(\PP^r,d)$,
and as soon as $\PP^1_k$ carries $n\ge3$ distinct $k$-rational points a map
with those markings is automorphism-free; for $n\le2$, a general degree $d$
map with distinct $k$-rational markings has trivial automorphisms as well.
\end{proof}

\begin{proposition}\label{prop:converse}
Let $k$ be infinite, $d$ odd, $Q$ a quaternion algebra, $C=\SB(Q)$, and
$A=M_m(Q)$ of degree $r+1=2m$ with $2m\ge d+1$.  Let $E=\prod_iL_i$ be an
\'etale $k$-algebra whose field factors have even degree over $k$ and admit
$L_i$-points on $C$.  Then $\overline M_{0,E}(\SB(A),d)$ has a $k$-point in
the automorphism-free locus.  
\end{proposition}

\begin{proof}
The $d$-th Veronese $\PP^1\to\PP^d$ is $\PGL_2$-equivariant for the
$\Sym^d$ representation, hence twists: the cocycle defining $C$ produces
a $k$-morphism $C\to\SB(A_d)$ of degree $d$, where $A_d$ is the algebra
of degree $d+1$ with $[A_d]=d[Q]=[Q]$, i.e.\ $A_d\cong M_{(d+1)/2}(Q)$
($d+1$ is even).  A linear embedding
$\SB(M_{(d+1)/2}(Q))\hookrightarrow\SB(M_m(Q))$ exists whenever
$(d+1)/2\le m$.  The embedding is obtained from the module structure over $Q$, not from
representation theory of $\PGL_2$: writing $a:=(d+1)/2\le m$, the inclusion
of right $Q$-modules $Q^{a}\hookrightarrow Q^{m}$ as the first $a$
coordinates identifies $\SB(M_a(Q))$, the variety of right ideals of reduced
dimension one in $M_a(Q)=\mathrm{End}_Q(Q^a)$, with a closed subvariety of
$\SB(M_m(Q))$; over $\kbar$ it becomes the linear embedding
$\PP^{d}\hookrightarrow\PP^{r}$ of a coordinate subspace.   Composing gives
$f:C\to X$ of degree $d$ over $k$.  The markings exist by hypothesis, and for
$k$ infinite one may choose them disjoint.  The twisted Veronese is a closed
immersion, so an automorphism of the stable map $f$ would be an automorphism
of $C$ fixing $f$, hence the identity; the resulting stable map is therefore
automorphism-free, and remains so after adding the markings.
\end{proof}

Thus the Brauer criterion \eqref{eq:AdC} is sharp for $d$ odd: under the
hypotheses of Proposition \ref{prop:converse}, a twisted Veronese curve on
$\SB(M_m(Q))$ supplies an automorphism-free $k$-point. 

\begin{proof}[Proof of Theorem \ref{thmB}]
The relation $[A]=d\,[C]$ of \eqref{eq:AdC} and the consequences (i)--(iv)
drawn from it are Theorem \ref{thm:amitsur}; sharpness for $d$ odd, by a
twisted Veronese, is Proposition \ref{prop:converse}.
\end{proof}

\section{Proof of Theorem 3: the parity dichotomy}\label{sec:proofC}

This section proves Theorem \ref{thmC}.  The target is $\PP^r$ and only the
markings are twisted; the mechanism is Springer's theorem \cite{Spr} applied
to the domain conic of the universal stable map, which carries divisors of
degrees $n$ and $d$.  We use it in the form: a conic over a field $K$ with a
closed point of odd degree is split, equivalently a quaternion class in
$\Br(K)[2]$ killed by an odd degree extension is trivial \cite{Spr}; see
\cite{GS} for the restriction--corestriction argument that deduces the second
formulation from the first.

Take $A$ split, so $X=\PP^r$, and twist only by $E$.  The
splitting $k^{r+1}=k\oplus k^r$ used in the proof of
Theorem~\ref{thm:noname} is defined over $k$ and is untouched by the
twisting, so the whole argument descends:

\begin{proposition}\label{prop:twistednoname}
For $\mathrm{char}\,k=0$, $r\ge1$, $d\ge3$ and any degree $n$ \'etale
algebra $E$,
\[
  \overline M_{0,E}(\PP^r,d)\;\sim_k\;M_{0,E|d}\times\Aff^{r(d+1)}_k ,
\]
where $M_{0,E|d}$ is the moduli of a genus zero curve equipped with an
$E$-marking and a degree $d$ divisor.
\end{proposition}

\begin{corollary}\label{cor:parity}
Let $\mathrm{char}\,k=0$, $r\ge1$, $d\ge3$, and let $E$ be a degree $n$
\'etale $k$-algebra.
\begin{enumerate}[label=(\roman*)]
\item If $\gcd(n,d)$ is odd, then $\overline M_{0,E}(\PP^r,d)$ is stably
rational over $k$.
\item If in addition the marking data supply three $k$-rational sections ---
for instance if $E$ has three factors isomorphic to $k$ --- then
$\overline M_{0,E}(\PP^r,d)$ is $k$-rational.
\item If $n$ and $d$ are both even, the class
$\alpha\in\Br\big(k(M_{0,E|d})\big)[2]$ of the generic $\PGL_2$-torsor ---
equivalently, of the generic domain conic $\mathcal C$ --- obstructs the
trivialization used in (i), and need not vanish. 
\end{enumerate}
\end{corollary}

\begin{proof}
Over the generic point of $M_{0,E|d}$ the universal curve is a conic
$\mathcal{C}$ carrying a divisor of degree $n$ (the markings) and one of
degree $d$, hence a divisor class of degree $\gcd(n,d)$.

(i)  If $\gcd(n,d)$ is odd then $\mathcal C$ carries a class of odd degree,
so it is split by Springer's theorem \cite{Spr}; the generic $\PGL_2$-torsor is
therefore trivial and admits a rational section.  Trivializing it identifies
$M_{0,E|d}$, birationally, with a quotient of the $k$-rational
$\Res_{E/k}(\PP^1_E)\times\PP^d$ by a trivial torsor, so
$M_{0,E|d}\times\PGL_2$ is $k$-rational; feeding this into the birational
product formula of Proposition \ref{prop:twistednoname} gives that
$\overline M_{0,E}(\PP^r,d)\times\PGL_2$ is $k$-rational, which is stable
rationality.

(ii)  Three $k$-rational sections rigidify the source $\PGL_2$, as
in Proposition \ref{prop:claim22}, so no torsor is left to trivialize and the
birational model is itself $k$-rational.

(iii)  With $n$ and $d$ both even no class of odd degree is available, the
conic can be nonsplit, and its class is precisely the obstruction to
trivializing the generic torsor.
\end{proof}

\begin{proof}[Proof of Theorem \ref{thmC}]
The proof follows from the birational product formula in Proposition \ref{prop:twistednoname}, and Corollary
\ref{cor:parity}.
\end{proof}

\section{Proof of Theorem 4: the M\"obius component on $\overline M_{0,n}$}
\label{sec:moebius-component-title}
\label{sec:moebius-component}
\label{sec:proofD}

This section proves Theorem \ref{thmD}.  The target is now
$\overline M_{0,n}$ itself, where the general theory used in
\S\ref{sec:Pr}--\S\ref{sec:proofC} is unavailable: $\overline M_{0,n}$ is not
convex for $n\ge7$, so $\overline M_{0,m}(\overline M_{0,n},\beta)$ need be
neither smooth nor irreducible, and no birational model of the kind used above
is available.  The input is the M\"obius family of \S\ref{sec:model}.

Throughout this section $k$ is a field, $k_s$ a separable closure,
$\Gamma=\mathrm{Gal}(k_s/k)$, and $n\ge 5$.  We write
$X=\overline M_{0,n}$, and for $2\le |I|\le n-2$ we write $\delta_I=\delta_{I^c}$
for the boundary divisor whose generic point is a two-component curve
carrying the markings $I$ on one component.  For
$2\le j\le\lfloor n/2\rfloor$ put $B_j=\sum_{|I|=j}\delta_I$.
Recall $\Aut(X)=S_n$ for $n\ge 5$ \cite{BM}, so the twisted
$k$-forms of $X$ are classified by
$H^1(\Gamma,S_n)=\{\text{degree $n$ \'etale $k$-algebras}\}$; we write
$X_E=\overline M_{0,E}$ for the form attached to $E$.

\subsection{The M\"obius component}
\label{ssec:moebius}

Let $\calR^\circ\cong\PGL_2^{\,S}$ be the space of $n$-tuples of
degree one maps $\PP^1\to\PP^1$ of \S\ref{sec:universal} and for $g=(g_1,\dots,g_n)\in \mathcal{R}^o$ let
\[
  \mu_g:\PP^1\dashrightarrow \overline M_{0,n},\qquad
  t\longmapsto \big[\PP^1;\,g_1(t),\dots,g_n(t)\big].
\]
Since $\PP^1$ is a smooth curve and $\overline M_{0,n}$ is proper, $\mu_g$
extends uniquely to a morphism, again denoted $\mu_g$; we call its
image, and by abuse the map itself, a \emph{M\"obius curve}.  Two
$\PGL_2$-actions are visible: the diagonal action on the target,
$u\cdot(g_i)=(ug_i)$, which does not change the point of $\overline M_{0,n}$
being traced, and reparametrisation of the source,
$(g_i)\cdot h=(g_ih)$.  Quotienting by the first alone returns the space
$Q\calR$ of \S\ref{sec:unirationality}, of dimension $3n-3$, whose points are
\emph{parametrized} M\"obius curves; the locus appearing in Proposition
\ref{prop:moebius} below is the further quotient by the second, which we
write $Q\calR^{\mathrm{un}}=Q\calR/\PGL_2^{\mathrm{pre}}$ and which has
dimension $3n-6$.  The two are therefore not the same space, and only the
latter maps to $\overline M_{0,0}(\overline M_{0,n},\beta)$, whose points are
maps up to reparametrization.

\begin{proposition}\label{prop:moebius}
Assume $n\ge 5$ and let $g=(g_1,...,g_n)\in \mathcal{R}^o$ be general.
\begin{enumerate}
\item[(i)] $\mu_g$ meets $\delta_{ij}$ exactly at the two fixed points of
$g_j^{-1}g_i$, transversally, and meets no other boundary divisor.  Hence
the class $\beta:=\mu_{g*}[\PP^1]$ satisfies
\[
  \beta\cdot\delta_I=\begin{cases}2,&|I|\in\{2,n-2\},\\
  0,&3\le|I|\le n-3,\end{cases}
\]
and $\beta$ is independent of $g$ and invariant under $S_n$.
\item[(ii)] $-K_{\overline M_{0,n}}\cdot\beta=2n$, so the expected dimension of
$\overline M_{0,0}(\overline M_{0,n},\beta)$ is $3n-6$.
\item[(iii)] The M\"obius locus
$Q\calR^{\mathrm{un}}:=\PGL_2\backslash\calR^\circ/\PGL_2^{\mathrm{pre}}\subset\overline M_{0,0}(\overline M_{0,n},\beta)$
is irreducible of dimension $3n-6$, and $\mu_g$ is a free map for
\emph{every} $g$ whose M\"obius curve meets the interior $M_{0,n}$ --- in
particular for general $g$, and with no hypothesis on $\mathrm{char}\,k$.
\item[(iv)] Consequently $Q\calR^{\mathrm{un}}$ is dense in a unique
irreducible component $\mathcal{M}^{\mathrm{M\ddot ob}}$ of
$\overline M_{0,0}(\overline M_{0,n},\beta)$, this component has the expected
dimension, and $\overline M_{0,0}(\overline M_{0,n},\beta)$ is smooth at its
general point.
\item[(v)] $\mathcal{M}^{\mathrm{M\ddot ob}}$ is stable under
$\Aut(\overline M_{0,n})=S_n$.
\end{enumerate}
\end{proposition}

\begin{proof}
(i) The markings $i$ and $j$ collide at $t$ precisely when
$g_j^{-1}g_i$ fixes $t$; for general $g$ the element $g_j^{-1}g_i$ is
regular semisimple, so this happens at exactly two points, each a simple
zero of $g_i-g_j$, whence transversality.  For general $g$ no three of
the $g_i$ agree at any $t$, since $g_j^{-1}g_i$ and $g_k^{-1}g_i$ have no
common fixed point; so $\mu_g$ meets no $\delta_I$ with $|I|\ge 3$ and
$|I^c|\ge3$.  The boundary divisors span
$\mathrm{Pic}(\overline M_{0,n})\otimes\mathbb{Q}$, so these numbers determine
$\beta$, and they are manifestly permutation invariant.

(ii) With Keel's formula as above, $\beta\cdot B_2=2\binom{n}{2}=n(n-1)$
and $\beta\cdot B_j=0$ for $j\ge3$, so
\[
  -K\cdot\beta=-c_2\,n(n-1)
  =\Big(2-\tfrac{2(n-2)}{n-1}\Big)n(n-1)=2n .
\]
Then \eqref{eq:expdim} gives $(n-3)+2n-3=3n-6$.  

(iii) The stabiliser of a general $g$ under
$\PGL_2\times\PGL_2$ is trivial: $ug_ih=g_i$ for all $i$ forces
$g_i^{-1}ug_i=h^{-1}$ for all $i$, hence $u=h=1$ for general $g$.  So
$Q\calR^{\mathrm{un}}$ is irreducible of dimension $3n-6$.  For freeness we argue directly, rather than by quoting the general
principle that a dominating family of rational curves has free general
member; the M\"obius family is explicit enough to make the computation
shorter than the citation, and the computation is characteristic-free.

The tangent space to $\underline{\Mor}(\PP^1,X)$ at $\mu_g$ is
$H^0(\PP^1,\mu_g^*T_X)$, and differentiating the family
$\calR^\circ$ at $g$ gives a $k$-linear map
\[
  c:\ \mathfrak{pgl}_2^{\oplus n}\longrightarrow H^0(\PP^1,\mu_g^*T_X),
  \qquad \xi=(\xi_i)\longmapsto c(\xi) .
\]
Here $c(\xi)$ is computed over the dual numbers $k[\epsilon]/(\epsilon^2)$
rather than with an exponential map, which is unavailable in small
characteristic: $\xi_i\in\mathfrak{pgl}_2$ corresponds to the point
$(1+\epsilon\xi_i)g_i$ of $\PGL_2\big(k[\epsilon]/(\epsilon^2)\big)$ lying
over $g_i$, the tuple $((1+\epsilon\xi_i)g_i)$ is a
$k[\epsilon]/(\epsilon^2)$-point of $\calR^\circ$ over $g$, and $c(\xi)$ is
the induced first-order deformation of $\mu_g$.
Let $t\in\PP^1$ be any point with $\mu_g(t)\in M_{0,n}$, i.e.\ with the $n$
points $g_i(t)$ pairwise distinct; all but finitely many $t$ are of this
kind, and by hypothesis at least one is.  At such a $t$,
\[
  T_{X,\mu_g(t)}\;=\;\Big(\bigoplus_i T_{\PP^1,g_i(t)}\Big)\Big/\mathfrak{pgl}_2 ,
\]
the quotient by the infinitesimal diagonal action, and the composite
$\mathfrak{pgl}_2^{\oplus n}\to T_{X,\mu_g(t)}$ sends $(\xi_i)$ to the class
of $(\xi_i(g_i(t)))_i$.  Since $\PGL_2$ acts transitively on $\PP^1$, the
evaluation $\mathfrak{pgl}_2\to T_{\PP^1,p}$ is surjective for every $p$ and
every characteristic, so $\bigoplus_i T_{\PP^1,g_i(t)}$ is hit, and hence so
is the quotient.  Therefore the evaluation map
$H^0(\PP^1,\mu_g^*T_X)\otimes\calO\to\mu_g^*T_X$ is surjective at $t$.

Now write $\mu_g^*T_X\cong\bigoplus_j\calO(a_j)$ (Grothendieck \cite{Gro57}).  If some
$a_j<0$ then $H^0(\PP^1,\mu_g^*T_X)=\bigoplus_{a_j\ge0}H^0(\calO(a_j))$, so
the image of the evaluation map lies in the proper subbundle
$\bigoplus_{a_j\ge0}\calO(a_j)$ at \emph{every} point, contradicting
surjectivity at $t$.  Hence $a_j\ge0$ for all $j$: $\mu_g^*T_X$ is nef and
$\mu_g$ is free.

(iv) Freeness gives $H^1(\PP^1,\mu_g^*T_X)=0$, hence smoothness of the moduli
space at $[\mu_g]$ and equality of actual and expected dimension there; since
$\dim Q\calR^{\mathrm{un}}$ already equals the expected dimension,
$Q\calR^{\mathrm{un}}$ is dense in a single component.

(v) $S_n$ permutes the factors of $\mathcal{R}^o$ and commutes with both $\PGL_2$
actions, and $\beta$ is $S_n$-invariant by (i).
\end{proof}

\subsection{Twisted forms: descent of the M\"obius component}

Let $E$ be a degree $n$ \'etale $k$-algebra and $X_E$ the corresponding
twisted form of $\overline M_{0,n}$.  Because the cocycle takes values in
$S_n$, the twisted form of $\calR^\circ\cong\PGL_2^{\,S}$ is the Weil restriction
\[
  R_E:=\Res_{E/k}\big(\PGL_{2,E}\big),
\]
a smooth affine $k$-group of dimension $3n$, and both $\PGL_2$-actions
of \S\ref{ssec:moebius} are defined over $k$ with the split group
$\PGL_{2,k}$ (the diagonal subgroup and the source reparametrisations are
$S_n$-invariant, and $\Gamma$ acts trivially on them).

\begin{lemma}\label{lem:weil}
$R_E$ is a $k$-rational variety with $R_E(k)$ Zariski dense (for $k$
infinite).
\end{lemma}

\begin{proof}
$\PGL_2$ over any field is an open subvariety of
$\PP(M_2)=\PP^3$, hence rational; choose a dense open
$U\subset\PGL_{2,E}$ with $U$ isomorphic to a dense open of
$\mathbb{A}^3_E$.  Weil restriction preserves open immersions, so
$\Res_{E/k}(U)$ is open in $R_E$ and open in
$\Res_{E/k}(\mathbb{A}^3_E)=\mathbb{A}^{3n}_k$; it is dense in $R_E$
because over $k_s$ it becomes $\prod U\subset\prod\PGL_2$.  For density of
$R_E(k)$ we use this chart rather than rationality in the abstract: $k$ is
infinite, so the $k$-points of the nonempty open
$\Res_{E/k}(U)\subset\mathbb{A}^{3n}_k$ are Zariski dense in it (the
complement of finitely many hypersurfaces in $\mathbb{A}^{3n}_k$ cannot
contain all $k$-points), and $\Res_{E/k}(U)$ is dense in $R_E$; hence
$R_E(k)$ is dense in $R_E$.
\end{proof}

\begin{theorem}\label{thm:twisted}
Let $n\ge5$, let $E$ be a degree $n$ \'etale $k$-algebra, and let $k$ be
infinite.  Then:
\begin{enumerate}
\item[(i)] $\beta$ and $\mathcal{M}^{\mathrm{M\ddot ob}}$ descend: there
is a canonical $k$-form
$\mathcal{M}^{\mathrm{M\ddot ob}}_E\subset\overline M_{0,0}(X_E,\beta)$.
On the open locus $R_E^{\mathrm{gf}}\subset R_E$ where the
$\PGL_2\times\PGL_2$-action is free, the map
$R_E^{\mathrm{gf}}\to\mathcal{M}^{\mathrm{M\ddot ob}}_E$ is a geometric
quotient in the category of $k$-varieties, with dense image.
\item[(ii)] $\mathcal{M}^{\mathrm{M\ddot ob}}_E$ is unirational over $k$
and $\mathcal{M}^{\mathrm{M\ddot ob}}_E(k)$ is Zariski dense.
\item[(iii)] Every $g\in R_E(k)$ determines a $k$-morphism
$\mu_g:\PP^1_k\to X_E$; the resulting family of $k$-rational curves is
Zariski dense in $X_E$.
\item[(iv)] The evaluation map
$\mathrm{ev}:R_E\times\PP^1\dashrightarrow X_E$, $(g,t)\mapsto\mu_g(t)$,
is a dominant $k$-rational map from a $k$-rational variety.  In
particular $X_E$ is unirational over $k$ and $X_E(k)$ is Zariski dense.
\end{enumerate}
\end{theorem}

\begin{proof}
(i) is Proposition~\ref{prop:moebius}(i),(v) plus Galois descent: an
$S_n$-invariant class and an $S_n$-stable component descend along any
cocycle valued in $S_n$; the quotient presentation is the twist of
$Q\calR^{\mathrm{un}}=\PGL_2\backslash\calR^\circ/\PGL_2^{\mathrm{pre}}$.  The
locus $R_E^{\mathrm{gf}}$ is nonempty and defined over $k$ by the stabilizer
computation of Proposition \ref{prop:moebius}(iii); no globally defined
quotient morphism on all of $R_E$ is asserted.  (ii) follows from
Lemma~\ref{lem:weil}: a $k$-rational variety dominates
$\mathcal{M}^{\mathrm{M\ddot ob}}_E$, and $R_E(k)$ is dense and surjects
onto a dense set of $k$-points.  For (iii), the family
$Q\calR^{\mathrm{un}}\times\PP^1\dashrightarrow X$ is $S_n$-equivariant, hence
descends; for fixed $g\in R_E(k)$ the induced rational map
$\PP^1_k\dashrightarrow X_E$ extends to a morphism since $\PP^1_k$ is a
smooth curve and $X_E$ is proper.  Density holds because
$\mathrm{ev}$ is dominant and $R_E(k)$ is dense.  (iv) is then immediate,
and $X_E(k)\ne\emptyset$ is dense because each $\mu_g$ has source
$\PP^1_k$.
\end{proof}

\subsection{Degree one, and where the parity obstruction lives}

The one case in which the Kontsevich space of $\PP^1$ is completely
understood is degree one:
$\overline M_{0,n}(\PP^1,1)\cong\PP^1[n]$, the Fulton--MacPherson
compactification of the configuration space of $n$ points on $\PP^1$,
obtained from $(\PP^1)^n$ by blowing up the diagonals; the product of
the evaluation maps is the Fulton--MacPherson morphism
$\PP^1[n]\to(\PP^1)^n$, and the GIT quotient of its stable locus by
$\PGL_2$ is $\overline M_{0,n}$ \cite{FMac,HT,Swi}.
All of this is $S_n$-equivariant, so it twists.

\begin{proposition}\label{prop:degree1}
Let $E$ be a degree $n$ \'etale $k$-algebra.  Then:
\begin{enumerate}[label=(\roman*)]
\item the twisted form $\PP^1[n]_E$ of $\overline M_{0,n}(\PP^1,1)$ is a
$k$-rational variety;
\item the stabilisation morphism
\[
  \pi_E:\PP^1[n]_E\longrightarrow X_E
\]
is defined over $k$, and over the open locus $M_{0,E}$ it is a
$\PGL_2$-torsor.
\end{enumerate}
\end{proposition}

\begin{proof}
(i)  This is Proposition \ref{prop:kRatX}(a)(ii) applied to the target
$\PP^1$ --- that is, to $(\PP^1)^n$ with $n=1$ --- and the degree
$d_1=1$: for any \'etale marking algebra $E$ of degree $n$, the space
$\overline M_{0,E}(\PP^1,1)=\PP^1[n]_E$ is $k$-rational, the degree one
coordinate rigidifying the source $\PGL_2$ over $k$.  

(ii)  Both the Fulton--MacPherson morphism and the GIT quotient of its stable
locus by $\PGL_2$ are $S_n$-equivariant, so $\pi_E$ descends to $k$; over
$M_{0,E}$ the configuration has trivial stabiliser, so the action is free
there and $\pi_E$ is a $\PGL_2$-torsor.
\end{proof}

\begin{proof}[Proof of Theorem \ref{thmD}]
  The numerical data $\beta\cdot\delta_{ij}=2$,
$\beta\cdot\delta_I=0$ for $3\le|I|\le n-3$ and $-K\cdot\beta=2n$, the freeness
of the general M\"obius map, and the resulting generically smooth irreducible
component of $\overline M_{0,0}(\overline M_{0,n},\beta)$ of the expected
dimension $3n-6$, are Proposition \ref{prop:moebius}.  The $S_n$-stability of
that component, its descent to every twisted form $X_E$, its domination by the
$k$-rational group $\Res_{E/k}(\PGL_{2,E})$, and the resulting
$k$-unirationality of $X_E$ together with a Zariski-dense set of $k$-points
lying on explicit $k$-rational M\"obius curves, are Theorem \ref{thm:twisted}.
\end{proof}


\begin{thebibliography}{99}

\bibitem{AV02} D.~Abramovich, A.~Vistoli, \emph{Compactifying the space of stable maps}, J.\ Amer.\ Math.\ Soc.\ \textbf{15} (2002), no.~1, 27--75.

\bibitem{Am} S.~A.~Amitsur, \emph{Generic splitting fields of central simple algebras}, Ann.\ of Math.\ (2) \textbf{62} (1955), 8--43.

\bibitem{BL} P.~Belorousski, \emph{Chow rings of moduli spaces of pointed elliptic curves}, PhD thesis, University of Chicago, 1998; A.~Logan, \emph{The Kodaira dimension of moduli spaces of curves with marked points}, Amer.\ J.\ Math.\ \textbf{125} (2003), no.~1, 105--138.

\bibitem{Bi73} A.~Bia\l ynicki-Birula, \emph{Some theorems on actions of algebraic groups}, Ann.\ of Math.\ (2) \textbf{98} (1973), 480--497.

\bibitem{Bog} F.~A.~Bogomolov, \emph{Rationality of the moduli of hyperelliptic curves of arbitrary genus}, in Conference on Algebraic Geometry (Vancouver, 1984), CMS Conf.\ Proc.\ \textbf{6}, Amer.\ Math.\ Soc., Providence, RI, 1986, pp.~17--37.

\bibitem{BK85} F.~A.~Bogomolov, P.~I.~Katsylo, \emph{Rationality of some quotient varieties}, Mat.\ Sb.\ (N.S.) \textbf{126(168)} (1985), no.~4, 584--589.

\bibitem{BLR} S.~Bosch, W.~L\"utkebohmert, M.~Raynaud, \emph{N\'eron Models}, Ergebnisse der Mathematik und ihrer Grenzgebiete (3), vol.~21, Springer-Verlag, Berlin, 1990.

\bibitem{BM} A.~Bruno, M.~Mella, \emph{The automorphism group of $\overline M_{0,n}$}, J.\ Eur.\ Math.\ Soc.\ \textbf{15} (2013), no.~3, 949--968.

\bibitem{Cas} A.-M.~Castravet, \emph{The Cox ring of $\overline M_{0,6}$}, Trans.\ Amer.\ Math.\ Soc.\ \textbf{361} (2009), no.~7, 3851--3878.

\bibitem{CT} A.-M.~Castravet, J.~Tevelev, \emph{Rigid curves on $\overline M_{0,n}$ and arithmetic breaks}, in Compact Moduli Spaces and Vector Bundles, Contemp.\ Math.\ \textbf{564}, Amer.\ Math.\ Soc., Providence, RI, 2012, pp.~19--67 (arXiv:1109.6705).

\bibitem{CTS07} J.-L.~Colliot-Th\'el\`ene, J.-J.~Sansuc, \emph{The rationality problem for fields of invariants under linear algebraic groups (with special regards to the Brauer group)}, in Algebraic Groups and Homogeneous Spaces, Tata Inst.\ Fund.\ Res.\ Stud.\ Math.\ \textbf{19}, Tata Inst.\ Fund.\ Res., Mumbai, 2007, pp.~113--186.

\bibitem{Cos} I.~Coskun, \emph{The Kontsevich moduli spaces of stable maps}, lecture notes, MIT 18.727 (Spring 2006), available at \url{https://ocw.mit.edu/courses/18-727-topics-in-algebraic-geometry-intersection-theory-on-moduli-spaces-spring-2006/}.

\bibitem{Dol87} I.~Dolgachev, \emph{Rationality of fields of invariants}, in Algebraic Geometry (Bowdoin, 1985), Proc.\ Sympos.\ Pure Math.\ \textbf{46}, Part~2, Amer.\ Math.\ Soc., Providence, RI, 1987, pp.~3--16.

\bibitem{DO} I.~Dolgachev, D.~Ortland, \emph{Point sets in projective spaces and theta functions}, Ast\'erisque, no.~165 (1988).

\bibitem{DR15} A.~Duncan, Z.~Reichstein, \emph{Versality of algebraic group actions and rational points on twisted varieties}, J.\ Algebraic Geom.\ \textbf{24} (2015), no.~3, 499--530 (with an appendix containing a letter from J.-P.~Serre).

\bibitem{EH} D.~Eisenbud, J.~Harris, \emph{The Kodaira dimension of the moduli space of curves of genus $\ge23$}, Invent.\ Math.\ \textbf{90} (1987), no.~2, 359--387.

\bibitem{FM14} B.~Fantechi, A.~Massarenti, \emph{On the rigidity of moduli of curves in arbitrary characteristic}, Int.\ Math.\ Res.\ Not.\ IMRN 2017, no.~8, 2431--2463.

\bibitem{Fa} G.~Farkas, \emph{Birational aspects of the geometry of $\overline M_g$}, in Surveys in Differential Geometry, Vol.~XIV, 57--110, Int.\ Press, Somerville, MA, 2009.

\bibitem{FHR} M.~Florence, N.~Hoffmann, Z.~Reichstein, \emph{On the rationality problem for forms of moduli spaces of stable marked curves of positive genus}, Ann.\ Sc.\ Norm.\ Super.\ Pisa Cl.\ Sci.\ (5) \textbf{XXII} (2021), 1091--1104 (arXiv:1709.05696).

\bibitem{FR1} M.~Florence, Z.~Reichstein, \emph{The rationality problem for forms of $\overline M_{0,n}$}, Bull.\ Lond.\ Math.\ Soc.\ \textbf{50} (2018), no.~1, 148--158.

\bibitem{FMac} W.~Fulton, R.~MacPherson, \emph{A compactification of configuration spaces}, Ann.\ of Math.\ (2) \textbf{139} (1994), no.~1, 183--225.

\bibitem{FP} W.~Fulton, R.~Pandharipande, \emph{Notes on stable maps and quantum cohomology}, in Algebraic Geometry---Santa Cruz 1995, Proc.\ Sympos.\ Pure Math.\ \textbf{62}, Part 2, Amer.\ Math.\ Soc., Providence, RI, 1997, pp.~45--96.

\bibitem{GS} P.~Gille, T.~Szamuely, \emph{Central Simple Algebras and Galois Cohomology}, Cambridge Studies in Advanced Mathematics \textbf{101}, Cambridge Univ.\ Press, Cambridge, 2006.

\bibitem{Gro57} A.~Grothendieck, \emph{Sur la classification des fibr\'es holomorphes sur la sph\`ere de Riemann}, Amer.\ J.\ Math.\ \textbf{79} (1957), 121--138.

\bibitem{HM} J.~Harris, D.~Mumford, \emph{On the Kodaira dimension of the moduli space of curves}, Invent.\ Math.\ \textbf{67} (1982), no.~1, 23--88.

\bibitem{HT} B.~Hassett, Y.~Tschinkel, \emph{Integral points and effective cones of moduli spaces of stable maps}, Duke Math.\ J.\ \textbf{120} (2003), no.~3, 577--599.

\bibitem{HTZ} B.~Hassett, Y.~Tschinkel, Z.~Zhang, \emph{Rationality of forms of $\overline M_{0,n}$}, J.\ Lond.\ Math.\ Soc.\ \textbf{111} (2025), no.~2, Paper No.\ e70097.

\bibitem{Hir89} A.~Hirschowitz, \emph{La rationalit\'e des sch\'emas de Hilbert de courbes gauches rationnelles suivant Katsylo}, in Algebraic Curves and Projective Geometry (Trento, 1988), Lecture Notes in Math.\ \textbf{1389}, Springer, Berlin, 1989, pp.~87--90.

\bibitem{HK} Y.~Hu, S.~Keel, \emph{Mori dream spaces and GIT}, Michigan Math.\ J.\ \textbf{48} (2000), 331--348.

\bibitem{Ig60} J.-I.~Igusa, \emph{Arithmetic variety of moduli for genus two}, Ann.\ of Math.\ (2) \textbf{72} (1960), 612--649.

\bibitem{Kat} P.~I.~Katsylo, \emph{Rationality of orbit spaces of irreducible representations of $\mathrm{SL}_2$}, Izv.\ Akad.\ Nauk SSSR Ser.\ Mat.\ \textbf{47} (1983), no.~1, 26--36.

\bibitem{Kat84} P.~I.~Katsylo, \emph{Rationality of the field of invariants of reducible representations of the group $\mathrm{SL}_2$}, Moscow Univ.\ Math.\ Bull.\ \textbf{39} (1984), no.~5, 80--83.

\bibitem{Kee} S.~Keel, \emph{Intersection theory of moduli space of stable $n$-pointed curves of genus zero}, Trans.\ Amer.\ Math.\ Soc.\ \textbf{330} (1992), no.~2, 545--574.

\bibitem{KM} S.~Keel, J.~McKernan, \emph{Contractible extremal rays on $\overline M_{0,n}$}, in Handbook of Moduli, Vol.~II, Adv.\ Lect.\ Math.\ \textbf{25}, Int.\ Press, Somerville, MA, 2013, pp.~115--130 (arXiv:alg-geom/9607009).

\bibitem{KP} B.~Kim, R.~Pandharipande, \emph{The connectedness of the moduli space of maps to homogeneous spaces}, in Symplectic Geometry and Mirror Symmetry (Seoul, 2000), World Sci.\ Publ., River Edge, NJ, 2001, pp.~187--201 (arXiv:math/0003168).

\bibitem{Kol} J.~Koll\'ar, \emph{Rational Curves on Algebraic Varieties}, Ergebnisse der Mathematik und ihrer Grenzgebiete (3), vol.~32, Springer-Verlag, Berlin, 1996.

\bibitem{KolSB} J.~Koll\'ar, \emph{Severi--Brauer varieties; a geometric treatment}, arXiv:1606.04368.

\bibitem{Ko} M.~Kontsevich, \emph{Enumeration of rational curves via torus actions}, in The Moduli Space of Curves (Texel Island, 1994), Progr.\ Math.\ \textbf{129}, Birkh\"auser, Boston, MA, 1995, pp.~335--368.

\bibitem{Mas14} A.~Massarenti, \emph{The automorphism group of $\overline M_{g,n}$}, J.\ Lond.\ Math.\ Soc.\ (2) \textbf{89} (2014), no.~1, 131--150.

\bibitem{MM} A.~Massarenti, M.~Mella, \emph{On the automorphisms of Hassett's moduli spaces}, Trans.\ Amer.\ Math.\ Soc.\ \textbf{369} (2017), no.~12, 8879--8902.

\bibitem{Mumford} D.~Mumford, J.~Fogarty, F.~Kirwan, \emph{Geometric Invariant Theory}, 3rd ed., Ergebnisse der Mathematik und ihrer Grenzgebiete (2), vol.~34, Springer-Verlag, Berlin, 1994.

\bibitem{Pan} R.~Pandharipande, \emph{Intersections of $\mathbb Q$-divisors on Kontsevich's moduli space $\overline M_{0,n}(\PP^r,d)$ and enumerative geometry}, Trans.\ Amer.\ Math.\ Soc.\ \textbf{351} (1999), 1481--1505.

\bibitem{Spr} T.~A.~Springer, \emph{Sur les formes quadratiques d'indice z\'ero}, C.~R.\ Acad.\ Sci.\ Paris \textbf{234} (1952), 1517--1519.

\bibitem{Swi} D.~Swinarski, \emph{An elementary GIT construction of the moduli space of stable maps}, Illinois J.\ Math.\ \textbf{53} (2009), no.~3, 843--856 (arXiv:math/0604092).

\end{thebibliography}
\end{document}